%% file: main.tex
\documentclass[11pt]{article}
\usepackage[english]{babel}
\usepackage{geometry}
\usepackage{amsmath,amsthm,amssymb,amscd,amsfonts}
\usepackage[all,cmtip]{xy}
\usepackage{tikz-cd}
\usetikzlibrary{calc,babel}
\usepackage{mathrsfs}
\usepackage{pdfpages}
\usepackage{float}
\usepackage{xcolor}
\usepackage{suffix}
\usepackage{mathtools}
\usepackage{soul}
\usepackage{gensymb}
\usepackage{enumitem}
\usepackage{calc}
\usepackage{csquotes}
\usepackage{gensymb}
\usepackage{placeins}
\usepackage{verbatim}
\usepackage{textcomp}
\usepackage{listings}
\usepackage{stmaryrd}
\usepackage{comment}
\usepackage[colorlinks = true, linkcolor = black, urlcolor  = black, citecolor = black, anchorcolor = black]{hyperref}
\usepackage[noabbrev,nameinlink]{cleveref} %cleveref must be last package loaded

\usepackage{soul}

\newtheorem{theorem}{Theorem}[section]
\newtheorem{proposition}[theorem]{Proposition}
\newtheorem{corollary}[theorem]{Corollary}
\newtheorem{lemma}[theorem]{Lemma}

\theoremstyle{definition}
\newtheorem{definition}[theorem]{Definition}
\newtheorem{remark}[theorem]{Remark}
\newtheorem{example}[theorem]{Example}

\theoremstyle{plain}

\usepackage{chngcntr}
\counterwithin{figure}{section}

\numberwithin{equation}{section}

\crefname{section}{\S}{\S\S}
\crefname{appendix}{App.}{Apps.}

\Crefname{equation}{Eq.}{Eqs.}
\Crefname{theorem}{Thm.}{Thms.}
\Crefname{lemma}{Lemma}{Lemmas}
\Crefname{corollary}{Cor.}{Cors.}
\Crefname{proposition}{Prop.}{Props.}

\Crefname{definition}{Def.}{Defs.}
\Crefname{notation}{Not.}{Nots.}
\Crefname{conjecture}{Conj.}{Conjs.}
\Crefname{remark}{Rk.}{Rks.}
\Crefname{problem}{Prob.}{Probs.}
\Crefname{example}{Ex.}{Exs.}

\Crefname{figure}{Fig.}{Figs.}
\Crefname{tabular}{Tab.}{Tabs.}

\DeclareMathOperator\codim{codim} %Codimension

\DeclareMathOperator\rank{rank} %Rank of a matrix
\DeclareMathOperator\spn{span} %Span (NB: \spn not a typo, primitive cmd. in \multicolumn pkcg.) [\Span]
\SetSymbolFont{stmry}{bold}{U}{stmry}{m}{n}

\parskip=\medskipamount
\begin{document}

\title{The Cartan--K{\"a}hler theorem for exterior differential systems on transitive Lie algebroids}

\author{Sonja Hohloch, Tom Mestdag and Kenzo Yasaka\\[2mm]
	{\small Department of Mathematics,  University of Antwerp,}\\
	{\small Middelheimlaan 1, 2020 Antwerpen, Belgium}
}

\date{}

\maketitle

\begin{abstract}
The notion of an exterior differential system (on a manifold) has recently been extended to the setting of a Lie algebroid. Here, we further develop the theory and we present two versions of the Cartan--K{\"a}hler theorem in the case where the anchor map of the Lie algebroid is surjective. We give an illustrative example and, as a concrete application, we make use of our results in a specific case of the so-called invariant inverse problem of the calculus of variations.
\end{abstract}

\vspace{3mm}


\vspace{3mm}

%Inverse problems for (mechanical) systems of particles,

\maketitle

\section{Introduction}
\input{sections/introduction.tex}

\section{Preliminaries}\label{sec:preliminaries}
\input{sections/preliminaries.tex}

\section{Integral manifolds}\label{sec:intmanifolds}
\input{sections/intmanifolds.tex}

\section{Integral elements}\label{sec:eds}
\input{sections/eds.tex}

\section{The Cartan--K{\"a}hler theorem}\label{sec:ckthm}
\input{sections/cartankahler.tex}

\section{A simple example}\label{sec:applications}
\input{sections/applications.tex}

\section{The invariant inverse problem}\label{sec:inv-inverse}
\input{sections/inv-inverse.tex}

\section{Outlook}\label{sec:outlook}
\input{sections/outlook.tex}

\paragraph{Acknowledgements.} The authors thank the Research Fund of the University of Antwerp (BOF) for its support through the DOCPRO4 project 46954 ``Symmetry reduction and unreduction in mechanics and geometry".

\section*{\texorpdfstring{Appendix: Comparison with the approach in \cite{mestdag2026edsla}}{Appendix: Comparison with the approach in [21]}}
\input{sections/comparison.tex}

\end{document}

%% file: sections/introduction.tex
Since the early twentieth century, there has been an ever increasing trend to study systems of partial differential equations by means of a geometric approach. One such approach, advocated by e.g., {\'E}.\ Cartan \cite{cartan1922invariant,cartan1945eds}, is that of \emph{exterior differential systems (EDS)}. Here, the PDEs are represented by a collection $\mathcal{I}$ of differential forms on a manifold $\overline{M}$, satisfying the condition that $\mathcal{I}$ is closed under exterior differentiation.  The main goal is to compute \emph{integral manifolds} for a given EDS, which are submanifolds $M\subset \overline{M}$ satisfying the property that all forms in $\mathcal{I}$ vanish on $M$. Since any system of partial differential equations can be written as an EDS (with an independence condition), finding solutions to a system of PDEs can be given the desired geometric interpretation, namely, finding integral manifolds of its associated EDS. 

Since its conception, EDS has made many valuable contributions in the contexts of e.g.,\ the isometric embedding problem \cite{bryant1991exterior}, Cartan's equivalence method \cite{olver1995equivalence} and the inverse problem of the calculus of variations \cite{aldridge2003invprob,anderson1992invprob,do2016invprob}. Besides, one may also think of the Frobenius theorem as an application of EDS theory, in the special case that the EDS only consists of one-forms \cite{bryant1991exterior}: The differential form version of the Frobenius theorem states that the algebraic ideal generated by linearly independent 1-forms $\theta^1,\dots,\theta^s\in \Omega^1(\overline{M})$, which we denote by $\mathcal{I}=\langle \theta^1,\dots,\theta^s\rangle_{\mathrm{alg}}$, corresponds to an integrable distribution $\mathcal{D}=\bigcap_{i=1}^s \ker(\theta^i)\subset \mathrm{T}\overline{M}$ if and only if $\mathcal{I}$ is a differential ideal, i.e., that $\mathrm{d}\mathcal{I}\subset \mathcal{I}$.

%In the differential form version the statement of the Frobenius theorem is that the algebraic ideal generated by the linearly independent 1-forms $\theta^1,\dots,\theta^s$, which we denote by $\mathcal{I}=\langle \theta^1,\dots,\theta^s\rangle_{\mathrm{alg}}$, is a differential ideal, i.e., that $\mathrm{d}\mathcal{I}\subset \mathcal{I}$, which is precisely the requirement for the integrability of the EDS.

The theory of EDS aims to generalize the consideration of $1$-forms to forms of arbitrary degree. In that case, the Frobenius theorem no longer applies and we need a more general theorem. The analogue of distributions in the   Frobenius theorem are the so-called \emph{integral elements} of the EDS. While the condition that $\mathcal{I}=\langle \omega^1,\dots,\omega^s\rangle\subset \Omega^{\ast}(\overline{M})$ is a differential ideal is automatically assumed in the theory of EDS, there is also an additional compatibility condition that needs to hold for the integral elements. That additional condition is what is called \emph{ordinary}. The theorem which generalizes the Frobenius theorem in this setting is called the Cartan--K{\"a}hler theorem, which for EDS on manifolds says the following: {\emph{Let $\mathcal{I}\subset \Omega^{\ast}(\overline{M})$ be a {(real) analytic} differential ideal on a manifold $\overline{M}$. Let $E_{z}\subset \mathrm{T}_{z} \overline{M}$ be an ordinary integral element of $\mathcal{I}$. Then there exists an integral manifold $j\colon X\to \overline{M}$ which passes through $z$ whose tangent space at $z$ is $E_z$}} (cf.\ Chap.\ 3, \S~2, Cor. 2.3 of \cite{bryant1991exterior}).

This paper serves as an extension to our previous paper \cite{mestdag2026edsla}, where we have extended the theory of exterior differential systems from manifolds to \emph{Lie algebroids}. This extension was particularly useful in solving the invariant inverse problem of the calculus of variations (also studied in \cite{barbero2016invproblemliealgebroids,crampin2008invariantinv,muzsnay2005invariantinv}), namely: \emph{Does a $G$-invariant Lagrangian exist for a given $G$-invariant second-order system on a Lie group $G$?} We showed that the problem of finding a \emph{reduced multiplier matrix} for the second-order system can be formulated as an EDS problem on a very specific (transitive) Lie algebroid. In the current paper, we will outline techniques for determining the existence of integral manifolds for an EDS on a transitive Lie algebroid. In particular, we will prove two extensions of the previously mentioned Cartan--K{\"a}hler theorem.

In \Cref{sec:preliminaries}, we present the basic preliminaries, in an attempt to keep the paper self-contained. In particular, we give an introduction to Lie algebroids and define an exterior derivative operator on sections of the dual of the Lie algebroid. We will also recall the statement of the main existence and uniqueness theorem for analytic partial differential equations, called the Cauchy--Kowalevski theorem. The proof of the Cartan--K{\"a}hler theorem depends on this existence theorem. In \Cref{sec:intmanifolds}, we briefly revisit the ideas in \cite{mestdag2026edsla}, where we introduce prolongation Lie algebroids to define integral manifolds of a set of forms on a transitive Lie algebroid. Since our current approach deviates a bit from the one we used in \cite{mestdag2026edsla}, we have outlined these subtle differences in the Appendix.

In \Cref{sec:eds}, we extend the notion of integral elements. Similarly to the setting of EDS on manifolds, we may also view integral elements as infinitesimal versions of integral manifolds. In particular, we consider \emph{K{\"a}hler-regular} and (a generalization of) ordinary integral elements, which are fundamental in the formulation of the Cartan--K{\"a}hler theorem.

In \Cref{sec:ckthm}, we give the statement of the Cartan--K{\"a}hler theorem for real analytic transitive Lie algebroids, which tells us when we can extend a $p$-dimensional integral manifold to a $(p+1)$-dimensional one. In this generalized setting, the fact that the coefficients of the anchor map and the structure functions of such Lie algebroids are real analytic is crucial. We also present an extension of a corollary of the Cartan--K{\"a}hler theorem to such Lie algebroids, which, roughly speaking, shows that a so-called ``$l$-ordinary integral element" is the initial data required for an exterior differential system to be integrable. Arguably, this corollary is more useful than the Cartan--K{\"a}hler theorem, and we will also sometimes refer to it  by that name.

For instructive purposes, in \Cref{sec:applications} we demonstrate how the Cartan--K{\"a}hler theorem for Lie algebroids can be used to determine existence of integral manifolds in a simple setting. Specifically, we consider an EDS generated algebraically by $1$-forms on the Lie algebroid $\mathrm{T}\mathbb{R}^3\times \mathbb{R}^1\to \mathbb{R}^3$. We then investigate a more complicated application in \Cref{sec:inv-inverse}, inspired by the invariant inverse problem, whose EDS approach was discussed extensively in \cite{mestdag2026edsla}. In this case, the differential ideal includes a $2$-form as an algebraic generator, which makes it more difficult to determine the regularity properties of integral elements.

Finally, we conclude the paper with some suggestions for future work.

%% file: sections/preliminaries.tex
We start with the basics on the theory of Lie algebroids. One can refer to e.g., \cite{mackenzie1987algebroid,marle2008algebroid,kowalevski1875cauchy} for a more in-depth treatment. From now on, we denote the space of sections of a vector bundle $\tau\colon A\to M$ by $\Gamma(A)$. In particular, when $A$ is the tangent bundle of a manifold $M$, the collection of smooth vector fields is written as $\mathcal{X}(M)\coloneqq \Gamma(\mathrm{T}M)$.

\begin{definition}
An \emph{anchored vector bundle} is a vector bundle $\tau\colon A\to M$ over a manifold $M$, together with a linear bundle map $\rho\colon A\to \mathrm{T}M$ (we will refer to this map as the \emph{anchor map}).
\end{definition}

\begin{definition}
A \emph{Lie algebroid} is an anchored vector bundle, together with a Lie bracket $\llbracket \cdot,\cdot\rrbracket$ on $\Gamma(A)$, such that
\begin{equation}\label{eqn:leibniz}
    \llbracket \sigma_1,f\sigma_2\rrbracket=\rho(\sigma_1)(f)\sigma_2+f\llbracket \sigma_1,\sigma_2\rrbracket
\end{equation}
for $\sigma_1,\sigma_2\in \Gamma(A)$ and $f\in C^{\infty}(M)$. Here it is used that $\rho$ descends to a map between sections $\rho\colon \Gamma(A)\to \mathcal{X}(M)$ (which we also denote by $\rho$). We say that a Lie algebroid is \emph{transitive} if its anchor map $\rho$ is surjective.
\end{definition}

The tangent bundle $\mathrm{T}M$ of a manifold $M$ is a fundamental example of a Lie algebroid (with the Lie bracket of vector fields and $\rho=\mathrm{id}$). Another example is a Lie algebra (taking $M=\{\ast\}$ to be a point, and $\rho=0$).

If $A$ is transitive, the fibers of the anchor map form a Lie algebra bundle $\ker(\rho)$ over $M$. One obtains a short exact sequence of vector bundles
\begin{equation*}
    0\longrightarrow \ker(\rho)\longrightarrow A\stackrel{\rho}{\longrightarrow} \mathrm{T}M\longrightarrow 0.
\end{equation*}

We will mainly be interested in transitive Lie algebroids in the real analytic category.
\begin{definition}
\cite{fernandes2025analytic,jiang2025analytic,zung2003analytic} A \emph{(real) analytic Lie algebroid} is an anchored vector bundle in which $A$ is a real analytic vector bundle, $\rho$ is a real analytic bundle map, together with a Lie bracket $\llbracket \cdot,\cdot\rrbracket$ on the sheaf of real analytic sections $\mathcal{A}$, such that \eqref{eqn:leibniz} holds for $\sigma_1,\sigma_2\in \mathcal{A}$ and all real analytic functions $f$ on $M$.
\end{definition}

Henceforth, we will simply write `analytic' for `real analytic' for all objects of interest (manifolds, bundles, Lie algebroids, functions, sections and forms, etc.).

Locally, if $\{e_{\alpha}\}$ is a basis of sections of $A$, one can write
\begin{equation*}
    \rho(e_{\alpha})=\rho_{\alpha}^i(x)\frac{\partial}{\partial x^i}\in \mathcal{X}(M),
\end{equation*}
\begin{equation*}
    \llbracket e_{\alpha},e_{\beta}\rrbracket=L_{\alpha \beta}^{\gamma}(x)e_{\gamma}.
\end{equation*}
It is easy to see that $L_{\alpha \beta}^{\gamma}=-L_{\beta \alpha}^{\gamma}$. The local functions $\{L_{\alpha \beta}^\gamma\}$ are called the \emph{structure functions}. On an analytic Lie algebroid, if $\{e_{\alpha}\}$ is a basis of analytic sections, then the functions $\rho_{\alpha}^i$ and $L_{\alpha \beta}^\gamma$ are analytic.

We call sections of exterior powers of the dual of the Lie algebroid \emph{forms} on a Lie algebroid. We denote the set of $k$-forms on $\tau\colon A\to M$ by $\Lambda^k(A)$.

\begin{definition}
The Lie algebroid \emph{exterior derivative} $\delta\colon \Lambda^k(A)\longrightarrow \Lambda^{k+1}(A)$ is defined by
\begin{multline*}
    \delta \omega(\sigma_1,\dots,\sigma_{k+1})=\sum_{i=1}^{k+1} (-1)^{i+1}\rho(\sigma_i)\left(\omega(\sigma_1,\dots,\hat{\sigma}_i,\dots,\sigma_{k+1})\right)\\+\sum_{1\leq i<j\leq k+1} (-1)^{i+j} \omega(\llbracket \sigma_i,\sigma_j\rrbracket,\sigma_1,\dots,\hat{\sigma}_i,\dots,\hat{\sigma}_j,\dots,\sigma_{k+1}),
\end{multline*}
for any $\sigma_1,\dots,\sigma_{k+1}\in \Gamma(A)$ and $\omega\in \Lambda^k(A)$. In the above expression, a hat denotes the omission of that symbol. In particular, for $f\in C^{\infty}(M)=\Lambda^0(A)$:
\begin{equation*}
    \delta f(\sigma)=\rho(\sigma)f.
\end{equation*}
\end{definition}
It is not difficult to verify that $\delta^2=0$ and that if $\omega\in \Lambda^k(A)$ and $\theta\in \Lambda^l(A)$, then
\begin{equation*}
    \delta(\omega\wedge \theta)=\delta \omega\wedge \theta+(-1)^k \omega\wedge \delta \theta.
\end{equation*}
If $\{e^{\alpha}\}$ denotes the dual basis of $\{e_{\alpha}\}$, then it can be seen that
\begin{equation*}
    \delta f=\frac{\partial f}{\partial x^i}\rho_{\alpha}^i e^{\alpha},\quad \delta e^{\alpha}=-\frac{1}{2}L_{\beta \gamma}^{\alpha}\, e^{\beta}\wedge e^{\gamma}.
\end{equation*}

We now recall the main local existence and uniqueness theorem for analytic partial differential equations, namely, the \emph{Cauchy--Kowalevski theorem}. Suppose that $F$ is a vector-valued function in $n+1$ independent variables $x^1,\dots,x^n,y$ and let $\phi$ be a vector-valued function in $x^1,\dots,x^n$. Let $x=(x^1,\dots,x^n)$ and consider the Cauchy problem
\begin{equation*}
    \frac{\partial F}{\partial y}=G\left(x,y,F,\frac{\partial F}{\partial x}\right),
\end{equation*}
subject to the boundary conditions
\begin{equation*}
    F(x,0)=\phi(x).
\end{equation*}
Then if both $G$ and $\phi$ are analytic in a neighborhood of $0$, the Cauchy problem has a unique analytic solution near $0$.

Finally, we will need the following lemma in the proof of the Cartan--K{\"a}hler theorem.

\begin{lemma}\label{lem:hadamard}
Let $S$ be an $s$-dimensional embedded submanifold of an $n$-dimensional manifold $M$ and let $f$ be an analytic function on $M$ such that $f|_S=0$. Pick a point $p\in S$. Then there exists an adapted coordinate chart $(U,\varphi)=(U,x^1,\dots,x^n)$ of $M$ around $p$ relative to $S$ (i.e., $U\cap S=\{(x^1,\dots,x^n)\mid x^{s+1}=0,\dots,x^n=0\}$), and analytic functions $g_{s+1},\dots,g_n$ on $U$ such that
\begin{equation*}
    f(x^1,\dots,x^n)=\sum_{i=s+1}^n g_i(x^1,\dots,x^n) x^i
\end{equation*}
on $U$.
\end{lemma}
\begin{proof}
We adapt the proof of Hadamard's lemma as presented in Chap. 2 of \cite{nestruev2020hadamard}. Let $(U,\varphi)$ be an adapted chart of $M$ relative to $S$ centered at $p$, with $\varphi(U)$ convex. Then
\begin{align*}
    f(x^1,\dots,x^n)&=f(x^1,\dots,x^n)-f(x^1,\dots,x^s,0,\dots,0)\\&=\sum_{i=s+1}^n f(x^1,\dots,x^i,0\dots,0)-f(x^1,\dots,x^{i-1},0,\dots,0)\\&=\sum_{i=s+1}^n \int_0^1 \frac{\mathrm{d}}{\mathrm{d}t}f(x^1,\dots,tx^i,0,\dots,0)\, \mathrm{d}t\\&=\sum_{i=s+1}^n x^i \int_0^1 \partial_i f(x^1,\dots,tx^i,0,\dots,0)\, \mathrm{d}t\\&=\sum_{i=s+1}^n x^i g_i(x^1,\dots,x^n),
\end{align*}
where
\begin{equation*}
    g_i(x^1,\dots,x^n)\coloneqq \int_0^1 \partial_i f(x^1,\dots,tx^i,0,\dots,0)\, \mathrm{d}t. \qedhere
\end{equation*}
\end{proof}

%% file: sections/intmanifolds.tex
In this section, we discuss prolongations of Lie algebroids. From this, we will be able to define integral manifolds of a set of forms on a transitive Lie algebroid. The current setup differs from the one we had in \cite{mestdag2026edsla}. For readers already familiar with that paper, we have compared the two approaches in the Appendix at the end of the paper.

Let $B\to N$ be a Lie algebroid with anchor $\rho_B$ and bracket $\llbracket \cdot,\cdot\rrbracket_B$, and let $f\colon N'\to N$ be a smooth map. Consider the manifold
\begin{equation*}
    \mathcal{L}^f B\coloneqq \rho_B^{\ast}(\mathrm{T}N')=\{(b_{n},v'_{n'})\in B\times \mathrm{T}N'\mid \rho_B(b_{n})=\mathrm{T}f(v'_{n'})\},
\end{equation*}
together with the projection $\tau^f\colon \mathcal{L}^f B\to N'$, $(b_n,v'_{n'})\mapsto n'$, where by construction $f(n')=n$. Note that
\begin{equation*}
    (\tau^f)^{-1}(n')=\rank(B)+\dim(N')-\dim(\rho_B(B_{f(n')})+(\mathrm{T}_{n'}f)(\mathrm{T}_{n'}N')),
\end{equation*}
hence $\tau^f$ takes the structure of a vector bundle only if $\dim(\rho_B(B_{f(n')})+(\mathrm{T}_{n'}f)(\mathrm{T}_{n'}N'))$ is constant. In particular, this holds if $B$ is a transitive Lie algebroid (i.e., $\rho_B$ is surjective), or if $f$ is a submersion. In either of these cases, $\tau^f$ is in fact, a Lie algebroid (see \cite{deleon2005survey,martinez2002prolongation} for details). The second projection serves as the anchor map:
\begin{equation}\label{eqn:anchorprolongation}
    \rho^f\colon \mathcal{L}^f B\longrightarrow \mathrm{T}N',\quad (b_n,v'_{n'})\longmapsto v'_{n'}.
\end{equation}
We now construct its bracket. To this end, we explain what projectable sections of $\tau^f$ are. A section of $\tau^f$ is of the form $(\sigma,X)$, where $\sigma$ is a section of the pullback bundle $f^{\ast}B\to N'$, and $X$ is a vector field on $N'$. We say that $(\sigma,X)$ is \emph{projectable} if there exists a section $s\in \Gamma(B)$ such that $\mu^f\circ (\sigma,X)=s\circ f$, where $\mu^f$ is the first projection, i.e.,
\begin{equation*}
    \mu^f\colon \mathcal{L}^f B\longrightarrow B,\quad (b_n,v'_{n'})\longmapsto b_n.
\end{equation*}
For two projectable sections, we define the bracket on $\mathcal{L}^f B$ by
\begin{equation}\label{eqn:bracketprolongation}
    \llbracket (\sigma_1,X_1),(\sigma_2,X_2)\rrbracket (n')\coloneqq \left(\llbracket \sigma_1,\sigma_2\rrbracket_{B}(n),[X_1,X_2](n')\right).
\end{equation}
We can extend the definition of the bracket to all sections of $\tau^f$ by imposing the Leibniz identity.
\begin{definition}
The vector bundle $\tau^f\colon \mathcal{L}^f B\to N'$, together with the anchor \eqref{eqn:anchorprolongation}, and the bracket defined on projectable sections by \eqref{eqn:bracketprolongation}, is called the \emph{$f$-prolongation Lie algebroid of $B$}.
\end{definition}

We now consider a special case. Let $\bar{\tau}\colon \overline{A}\to \overline{M}$ be a transitive Lie algebroid of rank $N$ with anchor $\bar{\rho}$ and bracket $\llbracket \cdot,\cdot \rrbracket_{\overline{A}}$, and let $i\colon M\to \overline{M}$ be an immersion. We consider its $i$-prolongation Lie algebroid, which we denote by $A\coloneqq \mathcal{L}^i \overline{A}$, i.e.,
\begin{equation*}
    \tau\colon A\longrightarrow M,\quad (\bar{a}_{\bar{m}},v_m)\longmapsto m,
\end{equation*}
where
\begin{equation*}
    A=\bar{\rho}^*(\mathrm{T}M)  = \{ (\bar{a}_{\bar{m}},v_m)\in  \overline{A}\times \mathrm{T}M \mid \bar{\rho}(\bar{a}_{\bar{m}}) = \mathrm{T}i(v_m)\}.
\end{equation*}
By construction, $i(m)=\bar{m}$. We denote the anchor map and bracket of $A$ by $\rho\colon A\to \mathrm{T}M$ and $\llbracket \cdot,\cdot \rrbracket$ respectively. Note that $A$ has rank $L\coloneqq N+\dim(M)-\dim(\overline{M})$.

\begin{example}
Since we restrict ourselves to transitive Lie algebroids, we point out that there are numerous examples. The ones that are of particular interest to us include:
\begin{itemize}
    \item Tangent bundles and Lie algebras (as discussed in \Cref{sec:preliminaries}),
    \item If $\pi\colon P\to M$ is a principal $G$-bundle, then the \emph{Atiyah algebroid} $A=\mathrm{T}P\big\slash G$ is a transitive Lie algebroid over $M$, where $\rho$ is induced by the differential $\mathrm{T}\pi\colon \mathrm{T}P\to \mathrm{T}M$, and the Lie bracket is given by the standard Lie bracket of $G$-invariant vector fields. They are useful in the study of dynamical systems with symmetry (cf. e.g.,\ \S~6 of \cite{mestdag2026edsla} or \S~9 of \cite{deleon2005survey}),
    \item Let $B\to N$ be a transitive Lie algebroid, and let $f\colon N'\to N$ be a smooth map. Then the $f$-prolongation Lie algebroid of $B$ is also a transitive Lie algebroid (i.e., the anchor map $\rho^f$ is surjective). Indeed, let $v'_{n'}\in \mathrm{T}N'$. Then $\mathrm{T}f(v'_{n'})\in \mathrm{T}N$. Since $B$ is transitive, there exists $b_n\in B$ such that $\rho_B(b_n)=\mathrm{T}f(v'_{n'})$. Hence, for every $v'_{n'}\in \mathrm{T}N'$, there exists $(b_n,v'_{n'})\in \mathcal{L}^f B$, as required.
    \item The \emph{IP Lie algebroid}, and its $p$-prolongation (useful in the study of the time-dependent invariant inverse problem, see e.g., \S~7 of \cite{mestdag2026edsla} or \Cref{sec:inv-inverse} of this paper).
\end{itemize}
\end{example}
In what follows, it will be useful to construct a convenient basis of $\overline{A}$ and its $i$-prolongation (i.e., $A$). Following our discussion in \Cref{sec:preliminaries}, note that the following sequence of vector bundles is short exact:
\begin{equation*}
    0\longrightarrow \ker(\bar{\rho})\longrightarrow \overline{A}\stackrel{\bar{\rho}}{\longrightarrow} \mathrm{T}\overline{M}\longrightarrow 0.
\end{equation*}
Since the anchor map $\bar{\rho}\colon \overline{A}\to \mathrm{T}\overline{M}$ is surjective, this short exact sequence splits,  and we will assume throughout that the Lie algebroid is such that an analytic splitting exists. Suppose we denote a splitting of $\bar{\rho}$ by $\bar{\gamma}\colon \mathrm{T}\overline{M}\to \overline{A}$. Then $\overline{A}\simeq \ker(\bar{\rho})\oplus \bar{\gamma}(\mathrm{T}\overline{M})$. We may now choose an adapted basis of $\overline{A}$ as follows. Choose a basis $\{\bar{k}_a\}$ of $\ker(\bar{\rho})$, and let $(z^j)$ be coordinates on $\overline{M}$. By injectivity of $\bar{\gamma}$, it follows that $\left\{\bar{\gamma}_j\coloneqq\bar{\gamma}\left(\frac{\partial}{\partial z^j}\right)\right\}$ is a basis of $\bar{\gamma}(\mathrm{T}\overline{M})$. This extends to a basis $\{\bar{k}_a,\bar{\gamma}_j\}$ of $\overline{A}$. We denote its dual basis by $\{\bar{k}^a,\bar{\gamma}^j\}$.

From this, we will construct a basis of $A$ as follows.
\begin{definition}
Given a splitting $\bar{\gamma}$ of $\bar{\rho}$, we call the splitting $\gamma$ of $\rho$, defined by
\begin{equation*}
    \gamma(v_m)\coloneqq (\bar{\gamma}(\mathrm{T}i(v_m)),v_m),
\end{equation*}
the \emph{$(i,\bar{\gamma})$-induced splitting}.
\end{definition}
Note that $\gamma(v_m)\in A$, because
\begin{equation*}
    \bar{\rho}(\bar{\gamma}(\mathrm{T}i(v_m)))=\mathrm{T}i(v_m),
\end{equation*}
and that $\gamma$ is indeed a splitting because
\begin{equation*}
    \rho(\gamma(v_m))=\rho(\bar{\gamma}(\mathrm{T}i(v_m)),v_m)=v_m.
\end{equation*}

Now suppose that $(x^i)$ are coordinates on $M$. A basis of sections of $A$ is then given by $\{k_a,\gamma_i\}$, where
\begin{equation*}
    k_a\coloneqq (\bar{k}_a,0),\quad \gamma_i\coloneqq \gamma\left(\frac{\partial}{\partial x^i}\right)=\left(\bar{\gamma}\left(\mathrm{T}i\left(\frac{\partial}{\partial x^i}\right)\right),\frac{\partial}{\partial x^i}\right),
\end{equation*}
where $\gamma$ is the $(i,\bar{\gamma})$-induced splitting. Analogously, we denote its dual basis by $\{k^a,\gamma^i\}$.

Let $I\coloneqq \mu^i$, which we will call the \emph{$i$-induced map} be the first projection, i.e.,
\begin{equation*}
    I\colon A \longrightarrow \overline{A},\quad (\bar{a}_{\bar{m}},v_m) \longmapsto  \bar{a}_{\bar{m}}.
\end{equation*}
Then, the following diagram commutes:
\begin{equation*}
\begin{gathered}
\xymatrix{
   A \ar[d]_{\rho} \ar@<3pt>[r]^{I} \ar@/^-2pc/[dd]_{\tau}  & \overline{A} \ar[d]_{\bar{\rho}} \ar@/^2pc/[dd]^{\bar{\tau}} \\
   \mathrm{T}M \ar@<3pt>[r]^{\mathrm{T}i} \ar[d] & \mathrm{T}\overline{M} \ar[d] \\
   M \ar@<3pt>[r]^{i} & \overline{M}
}
\end{gathered}
\end{equation*}
\begin{proposition}\label{prop:I-injective}
$I$ is injective.
\end{proposition}
\begin{proof}
Suppose that $I(\bar{a}_{\bar{m}},v_m)=I(\bar{b}_{\bar{n}},w_n)$. Then $\bar{a}_{\bar{m}}=\bar{b}_{\bar{n}}$ by definition of $I$. It remains to prove that $v_m=w_n$. Since $A$ is a prolongation Lie algebroid, it follows that
\begin{equation*}
    \mathrm{T}i(v_m)=\bar{\rho}(\bar{a}_{\bar{m}})=\bar{\rho}(\bar{b}_{\bar{n}})=\mathrm{T}i(w_n),
\end{equation*}
but since $i$ is an immersion, it follows that $v_m=w_n$, as required.
\end{proof}
From $I$, we may also define a pullback operator $I^{\ast}\colon \Lambda^k(\overline{A})\to \Lambda^k(A)$, as follows. On functions $\bar f\in C^{\infty}(\overline{M})$ and one-forms $\bar\theta$ on $\overline{A}$, we set
\begin{equation*}
    (I^{\ast}{\bar f})(m) = \bar f(i(m)),\qquad (I^{\ast}{\bar \theta})_m(\bar{a}_{\bar{m}},v_m)= \bar\theta_{i(m)}(I(\bar{a}_{\bar{m}},v_m))=\bar{\theta}_{i(m)}(\bar{a}_{\bar{m}}),
\end{equation*}
which we extend to $k$-forms by imposing linearity of $I^{\ast}$ and the rule $I^{\ast}(\omega\wedge \theta)=I^{\ast}\omega\wedge I^{\ast}\theta$ for all $\omega\in \Lambda^k(\overline{A})$ and $\theta\in \Lambda^l(\overline{A})$.

Locally, if $i$ takes the form
\begin{equation*}
    i\colon M\longrightarrow \overline{M},\quad (x^i)\longmapsto (z^j=\bar{z}^j(x^i)),
\end{equation*}
then
\begin{equation*}
    I^{\ast}\bar{k}^a=k^a,\quad I^{\ast}\bar{\gamma}^j=\frac{\partial \bar{z}^j}{\partial x^i}\gamma^i.
\end{equation*}

We may now state the definition of an integral manifold of a set of forms on a (transitive) Lie algebroid. A subspace $\mathcal{I}\subset \Lambda^{\ast}(\overline{A})$ is called an \emph{algebraic ideal} if it decomposes as a direct sum of subspaces $\mathcal{I}^k\subset \Lambda^{k}(\overline{A})$, and is closed under the wedge product with arbitrary differential forms (i.e., $\mathcal{I}$ is a graded ideal of $\Lambda^{\ast}(\overline{A})$). If, in addition, it satisfies $\bar{\delta} \mathcal{I}\subset \mathcal{I}$, we say that $\mathcal{I}$ is a \emph{differential ideal}.

Suppose we are given a finite set of forms of arbitrary degree $\{\omega^1,\dots,\omega^n\}\subset \Lambda^{\ast}(\overline{A})$. Then the smallest algebraic ideal containing these forms is given by
\begin{equation*}
    \langle \omega^1,\dots,\omega^n\rangle_{\mathrm{alg}}\coloneqq \left\{\alpha^j\wedge \omega^j\mid \alpha^1,\dots,\alpha^n\in \Lambda^{\ast}(\overline{A})\right\},
\end{equation*}
and the smallest differential ideal containing these forms is given by
\begin{equation*}
    \langle \omega^1,\dots,\omega^n\rangle\coloneqq \left\{\alpha^j\wedge \omega^j+\beta^j\wedge \bar{\delta} \omega^j\mid \alpha^1,\dots,\alpha^n,\beta^1,\dots,\beta^n\in \Lambda^{\ast}(\overline{A})\right\}.
\end{equation*}
In particular, if the algebraic ideal generated by $\{\omega^1,\dots,\omega^n\}\subset \Lambda^{\ast}(\overline{A})$ is also a differential ideal then $\langle \omega^1,\dots,\omega^n\rangle=\langle \omega^1,\dots,\omega^n\rangle_{\mathrm{alg}}$.

\begin{definition}
We say that a submanifold $i\colon M\to \overline{M}$, or $i(M)\subset \overline{M}$ is an \emph{integral manifold} of a differential ideal $\mathcal{I}$ if for all $\theta\in \mathcal{I}$, we have that $I^{\ast}\theta=0$.
\end{definition}

We denote throughout this paper the collection of $k$-forms in a differential ideal $\mathcal{I}$ by $\mathcal{I}^k\coloneqq \mathcal{I}\cap \Lambda^k(\overline{A})$. For the rest of this paper, we assume that all differential ideals satisfy $\mathcal{I}^0=\varnothing$.

%% file: sections/eds.tex
In this section, we extend the concept of an integral element to transitive Lie algebroids. We refer to e.g.,\ \cite{bryant1982exterior,bryant1991exterior,kamran2005eds,mckay2017eds} for this well-known concept, in the standard theory of EDS on manifolds. Roughly speaking, integral elements can be viewed as infinitesimal versions of integral manifolds, in the sense that the defining property for an integral manifold is only required to hold on a fiber of $\bar{\tau}\colon \overline{A}\to \overline{M}$. We also consider special classes of integral elements, as they are useful in our extension of the Cartan--K{\"a}hler theorem.  

\subsection{Introduction}

Given a vector bundle $\pi\colon V\to X$, we denote its fiber at $x\in X$ by $V_x\coloneqq \pi^{-1}(\{x\})$. We define the \emph{$k$-th Grassmann bundle} of $V$ to be the vector bundle $G_k(V)\to X$ whose fibers at $x\in X$ are the $k$-th Grassmannians of $V_x$ (i.e., the set of all $k$-dimensional subspaces of $V_x$).

Let $\mathcal{I}$ be a differential ideal of a transitive Lie algebroid $\bar{\tau}\colon \overline{A}\to \overline{M}$. In this section, we fix a number $k\leq N= \rank(\overline{A})$.

\begin{definition}
Let $\bar{m}\in \overline{M}$ and let $E_{\bar m}$ be a $k$-dimensional linear subspace $E_{\bar m}\subset \overline{A}_{\bar{m}}$. We say that $(\bar{m},E_{\bar m})\in G_k(\overline{A})$ is an \emph{integral element} of $\mathcal{I}$ if
\begin{equation*}
    \theta(v_1,\dots,v_k)=0,
\end{equation*}
for all $\theta\in \mathcal{I}^k$ and $v_1,\dots,v_k\in E_{\bar m}$.
\end{definition}
In the following, most of the time, we will simply write $E_{\bar{m}}$ for $(\bar{m},E_{\bar{m}})\in G_k(\overline{A})$.

We denote the set of all $k$-dimensional integral elements by $V_k(\mathcal{I})$, i.e.,
\begin{equation*}
    V_k(\mathcal{I})\coloneqq \{E_{\bar m}\in G_k(\overline{A})\mid \theta_{E_{\bar m}}=0~\mathrm{for}~\mathrm{all}~\theta\in \mathcal{I}^k\},
\end{equation*}
where $\theta_{E_{\bar m}}$ is the restriction of $\theta|_{\bar{m}}$ to $E_{\bar m}$. One can then see the following results:

\begin{proposition}\label{prop:subspaceintelement}
Let $E_{\bar m}$ be a $k$-dimensional integral element of $\mathcal{I}$. Let $W_{\bar m}$ be a subspace of $E_{\bar m}$. Then $W_{\bar m}$ is also an integral element of $\mathcal{I}$.
\end{proposition}
\begin{proof}
Suppose for the sake of finding a contradiction that $W_{\bar m}\subset E_{\bar m}$ is an $l$-dimensional subspace of $E_{\bar m}$ which is not an integral element. Let $\{w_1,\dots,w_l\}$ be a basis of $W_{\bar m}$. Then there exists $\theta\in \mathcal{I}^l$ such that $\theta(w_1,\dots,w_l)\neq 0$. We now claim that there exists a form $\eta\in \Lambda^{k-l}(\overline{A})$ and $u_1,\dots,u_{k-l}\in E_{\bar{m}}$ such that $\eta\wedge \theta(u_1,\dots,u_{k-l},w_1,\dots,w_l)\neq 0$. Let $U_{\bar m}$ be a complementary subspace of $W_{\bar m}$ (i.e., we have $E_{\bar m}=U_{\bar m}\oplus W_{\bar m}$), and let $\{u_1,\dots,u_{k-l}\}$ be a basis of $U_{\bar m}$. It then extends to a basis $\{u_1,\dots,u_{k-l},w_1,\dots,w_l\}$ of $E_{\bar m}$. Let $\{u^1,\dots,u^{k-l},w^1,\dots,w^l\}$ be its dual basis and let $\eta=u^1\wedge \dots\wedge u^{k-l}$. Then
\begin{equation*}
    \eta\wedge \theta(u_1,\dots,u_{k-l},w_1,\dots,w_l)=\theta(w_1,\dots,w_l)\neq 0.
\end{equation*}
Since $\eta\wedge \theta\in \mathcal{I}^k$, this contradicts the assumption that $E_{\bar m}$ is an integral element of $\mathcal{I}$.
\end{proof}
The concept of an integral element and integral manifold, are in fact, related. To establish this relationship, we first prove the following lemma.
\begin{lemma}\label{lem:zeroform}
Let $\omega\in \Lambda^k(\overline{A})$. If $I^{\ast}\omega=0$ for every $k$-dimensional submanifold $i\colon M\to \overline{M}$, then $\omega=0$.
\end{lemma}
\begin{proof}
We claim that if $\omega\in \Lambda^k(\overline{A})$ is non-zero, then there exists a $k$-dimensional submanifold $i\colon M\to \overline{M}$ such that $I^{\ast}\omega\neq 0$. Locally, $\omega\in \Lambda^k(\overline{A})$ takes the form
\begin{equation*}
    \omega=\sum_{l=0}^k \sum_{\substack{b_1<\dots<b_{k-l} \\ j_1<\dots<j_l}}\omega_{b_1,\dots,b_{k-l},j_1,\dots,j_l} \bar{k}^{b_1}\wedge \dots\wedge \bar{k}^{b_{k-l}}\wedge \bar{\gamma}^{j_1}\wedge \dots\wedge \bar{\gamma}^{j_l},
\end{equation*}
where $\omega_{b_1,\dots,b_{k-l},j_1,\dots,j_l}$ are functions on $\overline{M}$. Since $\omega\neq 0$, there exists $(a_1,\dots,a_{k-l},i_1,\dots,i_l)$ such that for at least one point $\bar{m}\in \overline{M}$, $\omega_{a_1,\dots,a_{k-l},i_1,\dots,i_l}(\bar{m})\neq 0$. Define $i\colon M\to \overline{M}$ by
\begin{equation*}
    (x^i)\longmapsto (z^j=\bar{z}^j(x))
\end{equation*}
where we choose the functions $\bar{z}^j$ in the following way. Fix $\bar{z}^{i_m}(x)=x^{i_m}$ for $m\in \{1,\dots,l\}$ (so that $I^{\ast}\bar{\gamma}^{i_m}=\gamma^{i_m}$). Choose any $k-l$ elements in the set $\{1,\dots,\dim(\overline{M})\}\setminus \{i_1,\dots,i_l\}$ and denote them by $\{c_1,\dots,c_{k-l}\}$. Choose $\bar{z}^{c_u}(x)=x^{c_u}$ for $u\in \{1,\dots,k-l\}$. Finally, for $j\not\in \{i_1,\dots,i_l,c_1,\dots,c_{k-l}\}$, let $\bar{z}^j(x)=z^j(\bar{m})$ be constants. Then
\begin{equation*}
    I^{\ast} \bar{\gamma}^{i_m}=\gamma^{i_m},\quad I^{\ast}\bar{\gamma}^{c_u}=\gamma^{c_u},\quad I^{\ast}\bar{\gamma}^j=0,\quad I^{\ast}\bar{k}^{a}=k^a.
\end{equation*}
We then have that
\begin{equation*}
    I^{\ast}\omega=(\omega_{a_1,\dots,a_{k-l},i_1,\dots,i_l}\circ i)\, k^{a_1}\wedge \dots\wedge k^{a_{k-l}}\wedge \gamma^{i_1}\wedge \dots\wedge \gamma^{i_l}+\mathrm{extra~terms},
\end{equation*}
where the extra terms are linearly independent from $k^{a_1}\wedge \dots\wedge k^{a_{k-l}}\wedge \gamma^{i_1}\wedge \dots\wedge \gamma^{i_l}$. In particular, at $m\in M$ with $i(m)=\bar{m}$,
\begin{equation*}
    (I^{\ast}\omega)_m=\omega_{a_1,\dots,a_{k-l},i_1,\dots,i_l}(\bar{m})\, k^{a_1}\wedge \dots\wedge k^{a_{k-l}}\wedge \gamma^{i_1}\wedge \dots\wedge \gamma^{i_l}+\mathrm{extra~terms}.
\end{equation*}
Hence, we have found a $k$-dimensional submanifold $M\subset \overline{M}$ such that $I^{\ast}\omega\neq 0$.
\end{proof}

\begin{proposition}\label{prop:intmfdintelement}
A submanifold $i\colon M\to \overline{M}$ is an integral manifold of $\mathcal{I}$ if and only if $I(A_m)$ is an integral element of $\mathcal{I}$ based at $\bar{m} = i(m)$ for all $m\in M$.
\end{proposition}
\begin{proof}
Suppose that $i\colon M\to \overline{M}$ is an integral manifold. Since $I$ is injective, the dimension of $I(A_m)$ is $L$ (i.e., the rank of $\tau\colon A\to M$). Now let $\alpha\in \mathcal{I}^L$ and $m\in M$ be arbitrary. For each choice of $\alpha$ and $m$, let $v_1,\dots,v_L\in A_m$ be arbitrary. Then $I(v_1),\dots,I(v_L)\in I(A_m)$, and we have that
\begin{equation}\label{eqn:intelementcomp}
    0=I^{\ast}\alpha(v_1,\dots,v_L)=\alpha(I(v_1),\dots,I(v_L)).
\end{equation}
Since the choice of $v_1,\dots,v_L$ was arbitrary, it follows that $I(A_m)$ is an integral element of $\mathcal{I}$ for all $m\in M$.

Conversely, suppose that $I(A_m)$ is an integral element based at $i(m)$ for all $m\in M$, and consider an arbitrary $\theta\in \mathcal{I}$. If $\deg(\theta)>L$, then $I^{\ast}\theta=0$ for dimension reasons. If $\deg(\theta)=L$, then $I^{\ast}\theta=0$ by the same computation as in \eqref{eqn:intelementcomp}. We now consider the case where $l\coloneqq \deg(\theta)<L$. Let $j\colon N\to M$ be an $l$-dimensional submanifold of $M$. Let $J\colon \mathcal{L}^j A\to A$ be the $j$-induced map. Note that for $n\in N$:
\begin{equation*}
    (I\circ J)((\mathcal{L}^j A)_n)\subset I(A_{j(n)}).
\end{equation*}
By \Cref{prop:subspaceintelement}, we see that
\begin{equation*}
    (I\circ J)^{\ast}\theta(v_1,\dots,v_l)=\theta((I\circ J)(v_1),\dots,(I\circ J)(v_l))=0,
\end{equation*}
for all $v_1,\dots,v_l\in (\mathcal{L}^j A)_n$. Hence $(I\circ J)^{\ast}\theta=J^{\ast}(I^{\ast}\theta)=0$ for all $l$-dimensional submanifolds $j\colon N\to M$. It follows from \Cref{lem:zeroform} that $I^{\ast} \theta=0$. 
\end{proof}

In the statement of our main theorem (\Cref{thm:maintheorem}), we wish to extend a $p$-dimensional integral manifold to a $(p+1)$-dimensional one. The infinitesimal analogue of this is to extend a $k$-dimensional integral element $E_{\bar{m}}\in V_k(\mathcal{I})$ to a $(k+1)$-dimensional integral element $E_{\bar{m}}^+\in V_{k+1}(\mathcal{I})$ containing $E_{\bar{m}}$. To do this, it is useful to introduce the concept of a polar space.
\begin{definition}
Let $\{e_1,\dots,e_k\}$ be a basis of $E_{\bar{m}}\subset \overline{A}_{\bar{m}}$. We define the \emph{polar space} of $E_{\bar{m}}$ to be the vector space
\begin{equation*}
    H(E_{\bar{m}})\coloneqq \{v\in \overline{A}_{\bar{m}}\mid \varphi(v,e_1,\dots,e_k)=0~\mathrm{for}~\mathrm{all}~\varphi\in \mathcal{I}^{k+1}\}.
\end{equation*}
\end{definition}
Note that $E_{\bar{m}}\subset H(E_{\bar{m}})$. With the following result, we can see why this definition is useful.
\begin{proposition}\label{prop:polar}
Let $E_{\bar{m}}\in V_k(\mathcal{I})$. Then a $(k+1)$-dimensional subspace $E_{\bar{m}}^+\subset \overline{A}_{\bar{m}}$ containing $E_{\bar{m}}$ is an integral element of $\mathcal{I}$ if and only if it satisfies $E_{\bar{m}}^+\subset H(E_{\bar{m}})$.
\end{proposition}
\begin{proof}
Suppose that $E_{\bar{m}}^+=E_{\bar{m}}+\mathbb{R}v$ and let $e_1,\dots,e_k$ be a basis of $E_{\bar{m}}$. Then $E_{\bar{m}}^+\in V_{k+1}(\mathcal{I})$ if and only if $\varphi(v,e_1,\dots,e_k)=0$ for all $\varphi\in \mathcal{I}^{k+1}$, which holds if and only if $v\in H(E_{\bar{m}})$.
\end{proof}
The space of all extensions is a (possibly empty) projective space $\mathbb{P}\left(H(E_{\bar{m}})\big\slash E_{\bar{m}}\right)$. It is useful to determine the dimension of this projective space to determine whether extensions exist and, if they do, for understanding the size of the space of extensions. To this end, we define a function $r\colon V_k(\mathcal{I})\to \mathbb{Z}_{\geq -1}$ by
\begin{equation*}
    r(E_{\bar{m}})\coloneqq \dim(H(E_{\bar{m}}))-(k+1).
\end{equation*}
Note that $r(E_{\bar{m}})=-1$ if and only if $E_{\bar{m}}$ lies in no $(k+1)$-dimensional integral element of $\mathcal{I}$. When $r(E_{\bar{m}})\geq 0$, the space of $(k+1)$-dimensional integral elements of $\mathcal{I}$ which contain $E_{\bar{m}}$ is a real projective space of dimension $r(E_{\bar{m}})$.

\subsection{K{\"a}hler-regular integral elements and manifolds}

In order to define the notion of K{\"a}hler-regularity, we need a few preliminary definitions. If $\mathcal{F}\subset C^{\infty}(X)$ is a set of smooth functions on a smooth manifold $X$, we denote by $Z(\mathcal{F})\subset X$ the set of common zeros of the functions in $\mathcal{F}$.
\begin{definition}
\cite{bryant1991exterior} We say that $x\in Z(\mathcal{F})$ is an \emph{ordinary zero} of $\mathcal{F}$ if there exists a neighborhood $V\subset X$ of $x$ and a finite set of functions $f^1,f^2,\dots,f^q\in \mathcal{F}$ such that
\begin{enumerate}
    \item Their differentials $\mathrm{d}f^1,\dots,\mathrm{d}f^q$ are independent on $V$,
    \item $Z(\mathcal{F})\cap V$ is their set of common zeros, i.e.,
    \begin{equation*}
        Z(\mathcal{F})\cap V=\{y\in V\mid f^1(y)=\cdots=f^q(y)=0\}.
    \end{equation*}
\end{enumerate}
\end{definition}
Let $\Omega$ be any $k$-form on $\overline{A}$. Then we denote by $G_k(\overline{A},\Omega)$ the set of all $E_{\bar{m}}\in G_k(\overline{A})$ such that $\Omega_{E_{\bar{m}}}\neq 0$. If $\theta$ is any other $k$-form on $\overline{A}$, we define a function $\theta_{\Omega}$ on $G_k(\overline{A},\Omega)$ by
\begin{equation}\label{eqn:formfunction}
    \theta_{E_{\bar{m}}}=\theta_{\Omega}(E_{\bar{m}})\Omega_{E_{\bar{m}}},
\end{equation}
where we consider $\theta_{E_{\bar{m}}}$ and $\Omega_{E_{\bar{m}}}$ as $k$-forms on $\Lambda^k(E_{\bar{m}})$. This function is well-defined since $\Lambda^k(E_{\bar{m}})$ is $1$-dimensional with basis $\Omega_{E_{\bar{m}}}$.

Now consider the set of functions
\begin{equation*}
    \mathcal{F}_{\Omega}(\mathcal{I})\coloneqq \{\theta_{\Omega}\mid \theta\in \mathcal{I}^k\}
\end{equation*}
on the manifold $X=G_k(\overline{A},\Omega)$. Since $\theta_{\Omega}(E_{\bar{m}})=0$ if and only if $\theta_{E_{\bar{m}}}=0$, the set of common zeros of $\mathcal{F}_{\Omega}(\mathcal{I})$ is given by $V_k(\mathcal{I},\Omega)\coloneqq V_k(\mathcal{I})\cap G_k(\overline{A},\Omega)$.
\begin{definition}
An integral element $E_{\bar{m}}\in V_k(\mathcal{I})$ is called \emph{K{\"a}hler-ordinary} if there exists a $k$-form $\Omega$ with $\Omega_{E_{\bar{m}}}\neq 0$ such that $E_{\bar{m}}$ is an ordinary zero of $\mathcal{F}_{\Omega}(\mathcal{I})$. In other words, there exist $k$-forms $\beta^1,\dots,\beta^q\in \mathcal{I}^k$ such that $V_{k}(\mathcal{I})$ is defined in a neighborhood $V$ of $E_{\bar{m}}$ in $G_k(\overline{A},\Omega)$ to be the set of all $F_{\bar n}\in V$ such that
\begin{equation*}
    f^c(F_{\bar n})\coloneqq \beta^c(e_{1}(F_{\bar n}),\dots,e_{k}(F_{\bar n}))=0,
\end{equation*}
for a choice of frame $\{e_{1},\dots,e_{k}\}$ on $V$ for all $1\leq c\leq q$. Additionally, it is required that $f^c$ has linearly independent differentials on $V$ for all $1\leq c\leq q$. We denote by $V_k^{\mathrm{o}}(\mathcal{I})\subset V_k(\mathcal{I})$ the subspace of $k$-dimensional K{\"a}hler-ordinary integral elements.
\end{definition}
\begin{definition}
We say that an integral element $E_{\bar{m}}\in V_k(\mathcal{I})$ is \emph{K{\"a}hler-regular} if $E_{\bar{m}}$ is K{\"a}hler-ordinary, and if the function $r$ is locally constant on a neighborhood of $E_{\bar{m}}$ in $V_k^{\mathrm{o}}(\mathcal{I})$. 

%We use the notation $V_k^{\mathrm{r}}(\mathcal{I})$ to denote the subspace of $k$-dimensional K{\"a}hler-regular integral elements.

\end{definition}

We are now in position to define K{\"a}hler-ordinary and K{\"a}hler-regular integral manifolds. We say that a submanifold $i\colon M\to \overline{M}$ is a \emph{K{\"a}hler-ordinary} integral manifold if $I(A_m)\subset \overline{A}_{i(m)}$ is a K{\"a}hler-ordinary integral element for all $m\in M$. Similarly, we say that $i\colon M\to \overline{M}$ is a \emph{K{\"a}hler-regular} integral manifold if $I(A_m)\subset \overline{A}_{i(m)}$ is a K{\"a}hler-regular integral element for all $m\in M$. If $i\colon M\to \overline{M}$ is connected and K{\"a}hler-regular, we set
\begin{equation*}
    r(M)\coloneqq r(I(A_m)),
\end{equation*}
which is well-defined since $r$ is locally constant.

Later, it will be useful in the formulation of the corollary of the Cartan--K{\"a}hler theorem (\Cref{cor:maintheorem}) to assume the existence of an increasing sequence of K{\"a}hler-regular integral elements. To this end, we define an integral flag.

\begin{definition}
Let $l\geq 0$ be a non-negative integer. We say that a sequence of subspaces $(E_l)_{\bar{m}}\subset (E_{l+1})_{\bar{m}}\subset \dots\subset (E_{l+n})_{\bar{m}}\subset \overline{A}_{\bar{m}}$ is an \emph{$l$-integral flag} of $\mathcal{I}$ of length $n$ based at $\bar{m}$ if each $(E_{l+j})_{\bar{m}}$ is an integral element of dimension $l+j$ based at $\bar{m}$. If $l=0$, we simply say that this sequence is an \emph{integral flag} of $\mathcal{I}$ of length $n$ based at $\bar{m}$.
\end{definition}

\begin{definition}
An integral element $E_{\bar{m}}$ is said to be \emph{$l$-ordinary} if there exists an $l$-integral flag $(E_l)_{\bar{m}}\subset (E_{l+1})_{\bar{m}}\subset \cdots \subset (E_{l+n})_{\bar{m}}\subset \overline{A}_{\bar{m}}$ based at $\bar{m}$ with $E_{\bar{m}}=(E_{l+n})_{\bar{m}}$, satisfying the condition that $(E_{l+j})_{\bar{m}}$ is K{\"a}hler-regular for $j\leq n-1$. If $l=0$, we simply say that $E_{\bar{m}}$ is \emph{ordinary}.
\end{definition}

%% file: sections/cartankahler.tex
In this section, we prove two versions of the Cartan--K{\"a}hler theorem for transitive Lie algebroids, stated in \Cref{thm:maintheorem} and \Cref{cor:maintheorem}. From now on, we will no longer make any notational distinction between the exterior derivative $\delta$ on $A$ and $\bar{\delta}$ on $\overline{A}$ (we will denote both by $\delta$).
\begin{theorem}[Cartan--K{\"a}hler, first version]\label{thm:maintheorem}
Let $\mathcal{I}\subset \Lambda^{\ast}(\overline{A})$ be an analytic differential ideal of an analytic transitive Lie algebroid $\bar{\tau}\colon \overline{A}\to \overline{M}$. Let $i\colon M\to \overline{M}$ be an analytic $p$-dimensional connected, K{\"a}hler-regular integral manifold of $\mathcal{I}$. Assume that $r(M)$ is a non-negative integer and set $r\coloneqq r(M)$. 

Let $i_{3}\colon R\to \overline{M}$ be any analytic submanifold of codimension $r$ containing $M$ (this corresponds to an immersion $i_{12}\colon M\to R$) and satisfies
\begin{equation}\label{eqn:transverse}
    I_3((\mathcal{L}^{i_3}\overline{A})_{i_{12}(x)})+H(I(A_x))=\overline{A}_{i(x)}
\end{equation}
for all $x\in M$, where $I$ is the $i$-induced map, and $I_{3}$ is the $i_{3}$-induced map.

Then there exists a $(p+1)$-dimensional connected analytic integral manifold $X$ of $\mathcal{I}$ satisfying $M\subset X\subset R$.
\end{theorem}

The core ideas of the proof below stem from \cite{bryant1982exterior,bryant1991exterior,ivey2016cartan}. Note that in the case that $\overline{A}=\mathrm{T}\overline{M}$, the condition \eqref{eqn:transverse} simplifies to the usual transversality condition
\begin{equation*}
    \mathrm{T}_x R+H(\mathrm{T}_x M)=\mathrm{T}_x\overline{M},\quad \mathrm{for}~x\in M.
\end{equation*}
For a better comparison with the classical case, we have also retained some of the notation from \cite{bryant1991exterior}.
\begin{proof}
We wish to find a manifold $X$ and immersions $i_1\colon M\to X$ and $i_2\colon X\to R$ that suitably fit in the diagram below:
\begin{equation*}
    \begin{CD}
    A@>I_1>>\mathcal{L}^{i_{23}}\overline{A}@>I_2>>\mathcal{L}^{i_3}\overline{A}@>I_3>>\overline{A}\\
    @V\tau VV @VVV @VVV@VV{\bar{\tau}}V \\
    M@>>i_1>X@>>i_2>R@>>i_3>\overline{M}\\
    \end{CD}
\end{equation*}
Since the theorem is a local result, it suffices to prove existence in a neighborhood of a point $x_0\in M$. We will then proceed as follows: After constructing a candidate for the manifold $X$, we will show in \Cref{lem:part2,lem:part3} that this $X$ is indeed an integral manifold. In particular, we will show in \Cref{lem:part2} that it suffices to show that a certain finite set of forms vanish on a prolongation algebroid over $X$. Finally, in \Cref{lem:part3} we will show that these forms do indeed vanish.

First, consider the vector space $I_3((\mathcal{L}^{i_3}\overline{A})_{i_{12}(x)})\cap H(I(A_x))$ for each $x\in M$. It has constant dimension, given by
\begin{align*}
    \dim(I_3((\mathcal{L}^{i_3}\overline{A})_{i_{12}(x)})\cap H(I(A_x)))&=\dim(I_3((\mathcal{L}^{i_3}\overline{A})_{i_{12}(x)}))+\dim(H(I(A_{x})))-\dim(\overline{A}_{i(x)})\\&=(N-r)+(r+L+1)-N\\&=L+1,
\end{align*}
where we have used that $I$ and $I_3$ are injective by \Cref{prop:I-injective}. 

Let us now consider a fixed $x_0 \in M$ and denote $\bar{x}_0\coloneqq i(x_0)\in \overline{M}$. We also consider a splitting $\bar{\gamma}$ of $\bar\rho$. From the previous observation, we may choose local coordinates on $\overline{M}$ centered at $x_0$ of the form $(x^1,\dots,x^p,y,u^1,\dots,u^s,v^1,\dots,v^r)$, and a basis of $\overline{A}$ given by 
\begin{equation*}
    \left\{\bar{k}_a,\bar{\gamma}_{x^i}\coloneqq \bar{\gamma}\left(\frac{\partial}{\partial x^i}\right),\bar{\gamma}_{y}\coloneqq \bar{\gamma}\left(\frac{\partial}{\partial y}\right),\bar{\gamma}_{u^{\sigma}}\coloneqq \bar{\gamma}\left(\frac{\partial}{\partial u^{\sigma}}\right),\bar{\gamma}_{v^{\rho}}\coloneqq \bar{\gamma}\left(\frac{\partial}{\partial v^{\rho}}\right)\right\},
\end{equation*}
where $M$ is described by $y=u=v=0$, $R$ by $v=0$, and such that the polar space $H(I(A_x))$ is spanned by the vectors $\bar{k}_a,\bar{\gamma}_{x^i},\bar{\gamma}_y,\bar{\gamma}_{v^{\rho}}$. Indeed, since $I(A_x)\subset H(I(A_x))$, it follows that $\bar{k}_a,\bar{\gamma}_{x^i}\in H(I(A_x))$. Moreover, by \eqref{eqn:transverse}, it follows that $\bar{\gamma}_{v^{\rho}}\in H(I(A_x))$, since $\bar{\gamma}_{v^{\rho}}\not\in I_3((\mathcal{L}^{i_3}\overline{A})_{i_{12}(x)})$. Finally, since $\dim(H(I(A_x)))=r+L+1$, one chooses the $y$ and $u^{\sigma}$ coordinates so that $H(I(A_x))=\spn\{\bar{k}_a,\bar{\gamma}_{x^i},\bar{\gamma}_y,\bar{\gamma}_{v^{\rho}}\}$.

Recall that $p=\dim(M)$ and that $r=\codim(R)$, so that $s+1$ denotes $\dim(\overline{M})-p-r$. The distinct roles that the coordinates $y$ and $u^{\sigma}$ play in the construction of $X$ will become clear later. We will denote the dual basis by $\{\bar{k}^a,\bar{\gamma}^{x^i},\bar{\gamma}^y,\bar{\gamma}^{u^{\sigma}},\bar{\gamma}^{v^{\rho}}\}$.

The above basis is such that, at $x_0$, $\{\bar{k}_a ({\bar x}_0),\bar{\gamma}_{x^i}({\bar x}_0) \}$ forms a basis of the integral element $F_{\bar{x}_0}\coloneqq I(A_{x_0})\in V_L(\mathcal{I})$. We may now consider a neighborhood $U$ of $F_{\bar{x}_0}$ in $G_L(\overline{A})$ such that every $E_z\in U$ with basepoint $z\in \overline{M}$ has a basis of the form
\begin{equation}\label{eqn:Kabasis}
    \mathcal{K}_a(E_z)=\bar{k}_a+q_a(E_z)\bar{\gamma}_y+p_a^{\sigma}(E_z)\bar{\gamma}_{u^{\sigma}}+w_a^{\sigma}(E_z)\bar{\gamma}_{v^{\rho}},
\end{equation}
\begin{equation}\label{eqn:Xibasis}
    \mathcal{X}_i(E_z)=\bar{\gamma}_{x^i}+q_{i}(E_z)\bar{\gamma}_y+p_{i}^{\sigma}(E_z)\bar{\gamma}_{u^{\sigma}}+w_{i}^{\sigma}(E_z)\bar{\gamma}_{v^{\rho}},
\end{equation}
where, here, $\bar{k}_a$ is short for $\bar{k}_a(z)$, etc. Since the above determines $E_z$ within $U$, we can use this to give coordinates $(x,y,u,v,q,p,w)$ on $U$, centered at $F_{\bar{x}_0}$. 

Now, let $\kappa\in \mathcal{I}^{L+1}$ and $v\in \overline{A}_z$. With this, we may define a function on $U$, given by
\begin{equation*}
    \psi_{\kappa,v}(E_z)\coloneqq \kappa(v,\mathcal{K}_1(E_z),\dots,\mathcal{K}_{L-p}(E_z),\mathcal{X}_1(E_z),\dots,\mathcal{X}_p(E_z)).
\end{equation*}
In particular,
\begin{equation*}
    \psi_{\kappa,v}(F_{\bar{x}_0})=\kappa(v,\bar{k}_1,\dots,\bar{k}_{L-p},\bar{\gamma}_{x^1},\dots,\bar{\gamma}_{x^p}),
\end{equation*}
where, again, it is understood that $\bar{k}_1$ is short for  $\bar{k}_1(x_0)$, etc. The polar space of $F_{\bar{x}_0}$ is then given by
\begin{equation*}
    H(F_{\bar{x}_0})=\{v\in \overline{A}_{\bar{x}_0}\mid \psi_{\kappa,v}(F_{\bar{x}_0})=0~\mathrm{for}~\mathrm{all}~\kappa\in \mathcal{I}^{L+1}\}.
\end{equation*}
Since the codimension of $H(F_{\bar{x}_0})$ is $s$, there exist $s$ analytic $(L+1)$-forms $\kappa^1,\dots,\kappa^s\in \mathcal{I}^{L+1}$ such that
\begin{equation*}
    H(F_{\bar{x}_0})=\{v\in \overline{A}_{\bar{x}_0}\mid \psi_{\kappa^{\sigma},v}(F_{\bar{x}_0})=0~\mathrm{for}~1\leq \sigma\leq s\}.
\end{equation*}
By a suitable change of basis, we may even assume that
\begin{equation}\label{eqn:kappasigma}
    \psi_{\kappa^{\sigma},v}(F_{\bar{x}_0})=\bar{\gamma}^{u^{\sigma}}(v),
\end{equation}
for $1\leq \sigma\leq s$ and $v\in \overline{A}_{\bar{x}_0}$. Now that we have found an expression for the polar space at $F_{\bar{x}_0}$, it is natural to search for a similar expression at any $E_z\in V_L(\mathcal{I})\cap U$ (possibly shrinking $U$). By definition,
\begin{equation*}
    H(E_z)=\{v\in \overline{A}_z\mid \psi_{\kappa,v}(E_z)=0~\mathrm{for}~\mathrm{all}~\kappa\in \mathcal{I}^{L+1}\}.
\end{equation*}
We claim that we may write
\begin{equation}\label{eqn:polargeneral}
    H(E_z)=\{v\in \overline{A}_z\mid \psi_{\kappa^{\sigma},v}(E_z)=0~\mathrm{for}~1\leq \sigma\leq s\}.
\end{equation}
For a fixed $E_z$, first note that the equations $\psi_{\kappa^{\sigma},v}(E_z)=0$ depend linearly on $v=\alpha_a\bar{k}_a+\beta_i \bar{\gamma}_{x^i}+\delta \bar{\gamma}_y+\epsilon_{\sigma}\bar{\gamma}_{u^{\sigma}}+\zeta_{\rho}\bar{\gamma}_{v^{\rho}}$. Moreover, at $F_{\bar{x}_0}$, the equations are linearly independent, and they are given explicitly by $\epsilon_{\sigma}=0$ for $1\leq \sigma\leq s$.

We now claim that the equations $\psi_{\kappa^{\sigma},v}(E_z)=0$ remain linearly independent for any $E_z\in V_L(\mathcal{I})\cap U$. Note that the matrix representation of the equations at $F_{\bar{x}_0}$ is given by
\begin{equation*}
    \begin{bmatrix} 0 & 0 & \dots & 1 & 0 & \dots & 0 & 0 & \dots & 0 \\ 0 & 0 & \dots & 0 & 1 & \dots & 0 & 0 & \dots & 0 \\ \vdots & \vdots & \vdots & \vdots & \vdots & \vdots & \vdots & \vdots & \vdots & \vdots \\ 0 & 0 & \dots & 0 & 0 & \dots & 1 & 0 & \dots & 0 \end{bmatrix}.
\end{equation*}
This matrix contains an $s\times s$ minor matrix with determinant one. We may therefore choose the neighborhood $U$ so that the determinant of this specific $s\times s$ matrix remains non-zero, because the set of linear functionals $\ell_{E_z}^{\sigma}\colon \overline{A}_z\to \mathbb{R}$, $v\mapsto \psi_{\kappa^{\sigma},v}(E_z)$ depends continuously (in particular, analytically) on $E_z$. This implies linear independence of the equations for any $E_z\in V_L(\mathcal{I})\cap U$. By K{\"a}hler-regularity, we know that $\codim(H(E_z))=s$ on $V_L(\mathcal{I})\cap U$ (possibly shrinking $U$), so, indeed, \eqref{eqn:polargeneral} holds.

Next, we will construct a specific vector $\mathcal{Y}(E_z) \in H(E_z)$. To do so, we start with a general expression of the form
\begin{equation*}
    \mathcal{Y} (E_z)=a\bar{\gamma}_y+b^{\sigma}\bar{\gamma}_{u^{\sigma}}+c^{\rho}\bar{\gamma}_{v^{\rho}}.
\end{equation*}
Notice that the $s$ equations $\psi_{\kappa^{\sigma},\mathcal{Y}(E_z)}(E_z)=0$ are linear equations for $a,b,c$. Explicitly, they are given by
\begin{equation*}
    A^{\sigma}(E_z)a+B_{\tau}^{\sigma}(E_z)b^{\tau}+C_{\rho}^{\sigma}(E_z)c^{\rho}=0,
\end{equation*}
where
\begin{equation*}
    A^{\sigma}(E_z)\coloneqq \psi_{\kappa^{\sigma},\bar{\gamma}_y}(E_z),\quad B^{\sigma}_{\tau}(E_z)\coloneqq \psi_{\kappa^{\sigma},\bar{\gamma}_{u^{\tau}}}(E_z),\quad C^{\sigma}_{\rho}(E_z)\coloneqq \psi_{\kappa^{\sigma},\bar{\gamma}_{v^{\rho}}}(E_z).
\end{equation*}
By assumption, at $E_z=F_{\bar{x}_0}$, these equations are linearly independent and they reduce to $b^{\sigma}=0$ (by \eqref{eqn:kappasigma}). Therefore, shrinking $U$ if necessary, we may assume that the $s\times s$ matrix $B(E_z)=(B_{\tau}^{\sigma}(E_z))$ is invertible for all $E_z\in U$. Making the choice $a=1$ and $c^{\rho}=0$, we may find unique analytic functions $G^{\sigma}$ so that for each $E_z\in U$, the vector
\begin{equation*}
    \mathcal{Y}(E_z)=\bar{\gamma}_y+G^{\sigma}(E_z)\bar{\gamma}_{u^{\sigma}},
\end{equation*}
solves
\begin{equation}\label{eqn:kappasigma2}
    \kappa^{\sigma}(\mathcal{Y}(E_z),\mathcal{K}_1(E_z),\dots,\mathcal{K}_{L-p}(E_z),\mathcal{X}_1(E_z),\dots,\mathcal{X}_p(E_z))=0.
\end{equation}
Here, $G^{\sigma}$ are related to the matrices $A$ and $B$ via
\begin{equation*}
    G^{\sigma}=-(B^{-1})_{\tau}^{\sigma} A^{\tau}.
\end{equation*}
Since $E_z$ can be specified using the coordinates $(x,y,u,v,q,p,w)$, in what follows, we will regard $G^{\sigma}$ as functions of these variables.

We are now in position to construct a candidate for $X$, as follows: We seek $i_{23}\colon X\to \overline{M}$ of the form $(x,y)\mapsto z(x,y)\coloneqq (x,y,u=F(x,y),v=0)$ on which the forms $\kappa^{\sigma}$ vanish. If we are able to do so, the manifold $X$ we construct in this way will be, at the very least, an integral manifold of $\langle \kappa^1,\dots,\kappa^s\rangle\subset \mathcal{I}$ (and the remaining task is to show that it is also an integral manifold for all forms in $\mathcal{I}$). 

First, note that such an $X$ would have an $i_{23}$-prolongation of $\overline{A}$ given by
\begin{equation*}
    \mathcal{L}^{i_{23}}\overline{A}=\{(\bar{a}_{\bar{m}},v_x)\in \overline{A}\times \mathrm{T}X\mid \bar{\rho}(\bar{a}_{\bar{m}})=\mathrm{T}i_{23}(v_x)\},
\end{equation*}
and that $\mathcal{L}^{i_{23}}\overline{A}\simeq \ker(\hat{\rho})\oplus \hat{\gamma}(\mathrm{T}X)$, where $\hat{\gamma}\colon \mathrm{T}X\to \mathcal{L}^{i_{23}}\overline{A}$ is the $(i_{23},\bar{\gamma})$-induced splitting, i.e.,
\begin{equation*}
    \hat{\gamma}(v_x)\coloneqq (\bar{\gamma}(\mathrm{T}i_{23}(v_x)),v_x).
\end{equation*}
Let
\begin{equation*}
    X_i(x,y)\coloneqq \mathrm{T}i_{23}\left(\frac{\partial}{\partial x^i}\Big|_{(x,y)}\right)=\left(\frac{\partial}{\partial x^i}+\frac{\partial F^{\sigma}(x,y)}{\partial x^i}\frac{\partial}{\partial u^{\sigma}}\right)\Big|_{z(x,y)},
\end{equation*}
and
\begin{equation*}
    Y(x,y)\coloneqq \mathrm{T}i_{23}\left(\frac{\partial}{\partial y}\Big|_{(x,y)}\right)=\left(\frac{\partial}{\partial y}+\frac{\partial F^{\sigma}(x,y)}{\partial y}\frac{\partial}{\partial u^{\sigma}}\right)\Big|_{z(x,y)}.
\end{equation*}
Therefore, a basis of $(\mathcal{L}^{i_{23}}\overline{A})_{(x,y)}$ takes the form
\begin{equation*}
    \left\{\hat{k}_a(x,y)=(\bar{k}_a(z(x,y)),0), \hat{\gamma}(X_i(x,y)), \hat{\gamma}(Y(x,y))\right\}.
\end{equation*}
We want the construction of $X$ to be such that $I_{23}(\hat{k}_a)=\mathcal{K}_a(E_z)$, $I_{23}(\hat{\gamma}(X_i(x,y)))=\mathcal{X}_i(E_z)$ and $I_{23}(\hat{\gamma}(Y(x,y)))=\mathcal{Y}(E_z)$,  because they form a basis of $I_{23}((\mathcal{L}^{i_{23}}\overline{A})_{(x,y)})$, which would imply by \eqref{eqn:kappasigma2} and \Cref{prop:intmfdintelement} that $X$ is integral for $\langle \kappa^1,\dots,\kappa^s\rangle \subset \mathcal{I}$. 

Since $I_{23}(\hat{k}_a)=\bar{k}_a$, $I_{23}(\hat{\gamma}(X_i(x,y)))=\bar{\gamma}_{x^i}+\frac{\partial F^{\sigma}}{\partial x^i}\bar{\gamma}_{u^{\sigma}}$, $I_{23}(\hat{\gamma}(Y(x,y)))=\bar{\gamma}_y+\frac{\partial F^{\sigma}}{\partial y}\bar{\gamma}_{u^{\sigma}}$, this can only happen when  $F$ satisfies the system of PDEs
\begin{equation}\label{eqn:pde-ck}
    \frac{\partial F^{\sigma}}{\partial y}=G^{\sigma}\left(x,y,F,0,q_a=0,q_i=0,p_a=0,p_i=\frac{\partial F}{\partial x^i},w_a=0,w_i=0\right),
\end{equation}
together with the initial condition $F(x,0)=0$ (so that $X$ contains $M$). By the Cauchy--Kowalevski theorem (cf. \Cref{sec:preliminaries}), there exists a unique analytic solution $F$ satisfying this system. This defines a connected submanifold $X\subset R$ and therefore concludes its construction. For our convenience, in the following, we replace $u$ by the functions $u-F(x,y)$, so that we can assume that $X$ is described by the equations $u=v=0$.

Recall that, so far, we have only shown that $X$ is a connected analytic submanifold of dimension $p+1$ which satisfies $M\subset X\subset R$ and is an integral manifold of $\langle \kappa^1,\dots,\kappa^s\rangle\subset \mathcal{I}$ (but not necessarily all forms in $\mathcal{I}$). We must now show that $X$ is an integral manifold.

To this end, since $F_{\bar{x}_0}$ is K{\"a}hler-ordinary, we can let $\beta^1,\dots,\beta^q$ be a set of analytic $L$-forms in $\mathcal{I}$ so that the functions
\begin{equation*}
    f^c(E_z)=\beta^c(K_1(E_z),\dots,K_{L-p}(E_z),X_1(E_z),\dots,X_p(E_z)),
\end{equation*}
for $1\leq c\leq q$ have linearly independent differentials on $U$ and are such that
\begin{equation*}
    V_L(\mathcal{I})\cap U=\{E_z\in U\mid f^1(E_z)=f^2(E_z)=\cdots=f^q(E_z)=0\},
\end{equation*}
where we shrink $U$ more if necessary.

Let $F_{\bar{x}_0}^+\coloneqq I_{23}((\mathcal{L}^{i_{23}}\overline{A})_{i_1(x_0)})$. By construction, $\mathcal{Y}(F_{\bar{x}_0})\in H(F_{\bar{x}_0})$, and therefore
\begin{equation*}
    F_{\bar{x}_0}^+=F_{\bar{x}_0}+\mathbb{R}\cdot \mathcal{Y}(F_{\bar{x}_0})\subset H(F_{\bar{x}_0}).
\end{equation*}
By \Cref{prop:polar}, it follows that $F_{\bar{x}_0}^+\in V_{L+1}(\mathcal{I})$. In particular, $F_{\bar{x}_0}^+$ extends the K{\"a}hler-regular integral element $F_{\bar{x}_0}$.

We now claim (in the lemma below) that in some neighborhood $U^+$ of $F_{\bar{x}_0}^+$ in $G_{L+1}(\overline{A})$, the set of ordinary common zeros of the $(L+1)$-forms $\{\beta^c\wedge \bar{\gamma}^y\}_{1\leq c\leq q}\cup \{\kappa^{\sigma}\}_{1\leq \sigma\leq s}$ is $V_{L+1}(\mathcal{I})\cap U^+$. This claim is useful, because if we can show (later, see the text before \Cref{lem:part3}) that $\beta^c\wedge \bar{\gamma}^y|_{F_{i_{23}(x,y)}^+}=\kappa^{\sigma}|_{F_{i_{23}(x,y)}^+}=0$ for some integral element $F_{i_{23}(x,y)}^+$  (that we will define later for any generic $(x,y)\in X$), then \Cref{prop:intmfdintelement} implies that $X$ is, in fact, an integral manifold.

Let $\Omega=\bar{k}^1\wedge \dots\wedge \bar{k}^{L-p}\wedge \bar{\gamma}^{x^1}\wedge \dots\wedge \bar{\gamma}^{x^p}$ and set
\begin{equation*}
    \Omega^+=\Omega\wedge \bar{\gamma}^y,\qquad B^c=(\beta^c\wedge \bar{\gamma}^y)_{\Omega^+} \qquad\mbox{and}\qquad K^{\sigma}=(\kappa^{\sigma})_{\Omega^+}
\end{equation*}
(we define $B^c$ and $K^{\sigma}$ using \eqref{eqn:formfunction}).

\begin{lemma}\label{lem:part2}
There exists an open neighborhood $U^+\subset G_{L+1}(\overline{A},\Omega^+)$ of $F_{\bar{x}_0}^+$ such that $\{B^c,K^{\sigma}\}$ have linearly independent differentials on $U^+$ and that
\begin{equation}\label{eqn:commonzeros}
    V_{L+1}(\mathcal{I})\cap U^+=\{E_z^+\in U^+\mid B^c(E_z^+)=K^{\sigma}(E_z^+)=0~\mathrm{for}~\mathrm{all}~c~\mathrm{and}~\sigma\}.
\end{equation}
\end{lemma}
\begin{proof}
We will first show the equality of the sets in \eqref{eqn:commonzeros}. It is clear that $V_{L+1}(\mathcal{I})\cap U^+$ is contained in the set on the right-hand side. We hence need to show the reverse inclusion.

First, we claim that every $E_z^+\in G_{L+1}(\overline{A},\Omega^+)$ contains a unique $L$-dimensional subspace $\widetilde{E}\subset E_z^+$ on which $\bar{\gamma}^y$ vanishes. Indeed, if $\widetilde{E}'$ would be a second such $L$-dimensional subspace, then $\widetilde{E}+\widetilde{E}'=E_z^+$ and $\bar{\gamma}^y|_{E_z^+}=\bar{\gamma}^y|_{\widetilde{E}+\widetilde{E}'}=0$. But this would contradict the assumption that $\Omega^+|_{E_z^+}\neq 0$.

Choose any neighborhood $U^+\subset G_{L+1}(\overline{A},\Omega^+)$ of $F_{\bar{x}_0}^+$ such that $\widetilde{E}\in U$ whenever $E_z^+\in U^+$. One then has unique functions $A_a^{\sigma},B_a^{\rho},A_{i}^{\sigma},B_i^{\rho},a^{\sigma},b^{\rho}$ in $G_{L+1}(\overline{A},\Omega^+)$ such that the $L+1$ vectors
\begin{equation*}
    \kappa_a(E_z^+)=\bar{k}_a+A_{a}^{\sigma}(E_z^+)\bar{\gamma}_{u^{\sigma}}+B_{a}^{\rho}(E_z^+)\bar{\gamma}_{v^{\rho}},
\end{equation*}
\begin{equation*}
    \xi_i(E_z^+)=\bar{\gamma}_{x^i}+A_i^{\sigma}(E_z^+)\bar{\gamma}_{u^{\sigma}}+B_i^{\rho}(E_z^+)\bar{\gamma}_{v^{\rho}},
\end{equation*}
\begin{equation*}
    \eta(E_z^+)=\bar{\gamma}_y+a^{\sigma}(E_z^+)\bar{\gamma}_{u^{\sigma}}+b^{\rho}(E_z^+)\bar{\gamma}_{v^{\rho}},
\end{equation*}
form a basis of $E_z^+$. In particular, $\{\kappa_a(E_z^+),\xi_i(E_z^+)\}$ is a basis of $\widetilde{E}$. We thus obtain a coordinate system $(x,y,u,v,A,B,a,b)$ for $G_{L+1}(\overline{A},\Omega^+)$. Since $\bar{\gamma}^y(\kappa_a(E_z^+))=\bar{\gamma}^y(\xi_i(E_z^+))=0$ and $\bar{\gamma}^y(\eta(E_z^+))=1$, it follows that
\begin{equation}\label{eqn:Bc}
    B^c(E_z^+)=f^c(\widetilde{E}),
\end{equation}
for all $E_z^+\in G_{L+1}(\overline{A},\Omega^+)$ and $1\leq c\leq q$. Moreover, note that an explicit expression for $K^{\sigma}$ is given by 
\begin{align}\label{eqn:Ksigma}
    K^{\sigma}(E_z^+)&=\frac{\kappa^{\sigma}(\kappa_1(E_z^+),\dots,\kappa_{L-p}(E_z^+),\xi_1(E_z^+),\dots,\xi_p(E_z^+),\eta(E_z^+))}{\Omega^+(\kappa_1(E_z^+),\dots,\kappa_{L-p}(E_z^+),\xi_1(E_z^+),\dots,\xi_p(E_z^+),\eta(E_z^+))} \nonumber\\&=\kappa^{\sigma}(\kappa_1(E_z^+),\dots,\kappa_{L-p}(E_z^+),\xi_1(E_z^+),\dots,\xi_p(E_z^+),\eta(E_z^+)).
\end{align}
Now, suppose that $E_z^+\in \{E_z^+\in U^+\mid B^c(E_z^+)=K^{\sigma}(E_z^+)=0~\mathrm{for}~\mathrm{all}~c~\mathrm{and}~\sigma\}$. Then since $B^c(E_z^+)=0$, it follows from \eqref{eqn:Bc} that $\widetilde{E}\in V_L(\mathcal{I})\cap U$. Therefore, since $K^{\sigma}(E_z^+)=0$ as well, it follows from \eqref{eqn:Ksigma} that $\eta(E_z^+) \in  H(\widetilde{E})$. Therefore, $E_z^+\subset H(\widetilde{E})$, so it follows from \Cref{prop:polar} that $E_z^+\in V_{L+1}(\mathcal{I})\cap U^+$.

It remains to show that $\{B^c,K^{\sigma}\}$ have linearly independent differentials at $F_{\bar{x}_0}^+$ (having linearly independent differentials is an open property, so linear independence at $F_{\bar{x}_0}^+$ implies linear independence on a possibly smaller open neighborhood $U^+$). By \eqref{eqn:Ksigma}, we see that the functions $K^{\sigma}$ are affine in the coordinates $a,b$. Hence $K^{\sigma}=M_{\nu}^{\sigma}a^{\nu}+N^{\sigma}$ for some functions $M,N$ that only depend on $(x,y,u,v,A,B,b)$. We claim that at $F_{\bar{x}_0}^+$, we have $K^{\sigma}(F_{\bar{x}_0}^+)=0$. First, note that
\begin{equation}\label{eqn:kappasigma3}
    K^{\sigma}(F_{\bar{x}_0}^+)=\kappa^{\sigma}(\bar{k}_1,\dots,\bar{k}_{L-p},\bar{\gamma}_{x^1},\dots,\bar{\gamma}_{x^p},\eta(F_{\bar{x}_0}^+))=(-1)^L \bar{\gamma}^{u^{\sigma}}(\eta(F_{\bar{x}_0}^+))=(-1)^L a^{\sigma}(F_{\bar{x}_0}^+).
\end{equation}
Besides, since 
\begin{equation*}
    0=\kappa^{\sigma}(\mathcal{Y}(F_{\bar{x}_0}),\mathcal{K}_1(F_{\bar{x}_0}),\dots,\mathcal{K}_{L-p}(F_{\bar{x}_0}),\mathcal{X}_1(F_{\bar{x}_0}),\dots,\mathcal{X}_p(F_{\bar{x}_0}))=\bar{\gamma}^{u^{\sigma}}(\mathcal{Y}(F_{\bar{x}_0}))=G^{\sigma}(F_{\bar{x}_0}),
\end{equation*}
we get $G^{\sigma}(F_{\bar{x}_0})=0$. Therefore, $\mathcal{Y}(F_{\bar{x}_0})=\bar{\gamma}_y$, so it follows that $F_{\bar{x}_0}^+=\spn\{\bar{k}_a,\bar{\gamma}_{x^i},\bar{\gamma}_y\}$. From this we may conclude that $a^{\sigma}(F_{\bar{x}_0}^+)=0$ and therefore also $K^{\sigma}(F_{\bar{x}_0}^+)=0$. 

We may take $M_{\nu}^{\sigma}(F_{\bar{x}_0}^+)=(-1)^L \delta_{\nu}^{\sigma}$, and $N^{\sigma}(F_{\bar{x}_0}^+)=0$. Hence, we may assume (shrinking $U^+$ if necessary) that $(M_{\nu}^{\sigma}(E_z^+))$ is invertible for all $E_z^+\in U^+$. For this reason, there exist unique functions $T^{\nu}$ which depend only on $x,y,u,v,A,B,b$ such that $N^{\sigma}=-M_{\nu}^{\sigma}T^{\nu}$, and so we may write $K^{\sigma}=M_{\nu}^{\sigma}(a^{\nu}-T^{\nu})$. Note that the functions $a^{\sigma}-T^{\sigma}$ have independent differentials, and that the $B^c$ are independent from the $a^{\sigma}-T^{\sigma}$ (since $B^c$ only depends on $x,y,u,v,A,B$). Hence, $\{B^c,K^{\sigma}\}$ have linearly independent differentials at $F_{\bar{x}_0}^+$ if and only if $\{B^c\}$ have linearly independent differentials at $F_{\bar{x}_0}^+$.

Now, let $K\subset U$ be the set of $L$-dimensional subspaces on which $\bar{\gamma}^y$ vanishes. Using the coordinates $(x,y,u,v,q,p,w)$ defined by \eqref{eqn:Kabasis} and \eqref{eqn:Xibasis}, one sees that $K$ is a submanifold of $U$, defined by
\begin{equation*}
    K=\{E_z\in U\mid q_a(E_z)=q_i(E_z)=0\}.
\end{equation*}
Note that if $E_z^+\in U^+$, then $\widetilde{E}\in K$. Also, since the equality $B^c(E_z^+)=f^c(\widetilde{E})$ holds, it follows that $\{B^c\}$ have linearly independent differentials at $F_{\bar{x}_0}^+$ if and only if $\{f^c\}$ have linearly independent differentials at $F_{x_0}$, following their restrictions to $K$ (i.e., $(\mathrm{d}f^c)_{F_{\bar{x}_0}}|_{\mathrm{T}_{F_{\bar{x}_0}}K}$ are linearly independent). Equivalently,
\begin{equation*}
    \ker((\mathrm{d}f)_{F_{\bar{x}_0}})\cap \mathrm{T}_{F_{\bar{x}_0}}K=\{0\}.
\end{equation*}
Since $\ker((\mathrm{d}f)_{F_{\bar{x}_0}})=\mathrm{T}_{F_{x_0}}(V_L(\mathcal{I})\cap U)$ and $\ker((\mathrm{d}q)_{F_{\bar{x}_0}})=\mathrm{T}_{F_{\bar{x}_0}}K$, it follows that this is equivalent to requiring that $\{q_a,q_i\}$ have independent differentials on the set $V_L(\mathcal{I})\cap U$.

We now claim that this is indeed the case, by showing that they are already independent on a (yet to be defined) submanifold $\mathcal E$ of $V_L(\mathcal{I})\cap U$. Let $\boldsymbol{\lambda}=(\lambda_1,\dots,\lambda_{L-p})\in \mathbb{R}^{L-p}$ and $\boldsymbol{\mu}=(\mu_1,\dots,\mu_p)\in \mathbb{R}^p$ be two vectors consisting of real numbers. Since $F_{\bar{x}_0}^+\in V_{L+1}(\mathcal{I})$, it follows that for sufficiently small $\lambda_a$ and $\mu_i$ the $L$-dimensional subspace $F_{\boldsymbol{\lambda},\boldsymbol{\mu}}\subset F_{\bar{x}_0}^+$ spanned by the vectors
\begin{equation*}
    Q_{a}(\boldsymbol{\lambda})=\bar{k}_a+\lambda_{a}\bar{\gamma}_y,\quad Q_{i}(\boldsymbol{\mu})=\bar{\gamma}_{x^i}+\mu_{i}\bar{\gamma}_y,
\end{equation*}
lies in $V_L(\mathcal{I})\cap U$. This is based on identifying the coordinates $q_a$ and $q_i$ with $\lambda_a$ and $\mu_i$ and the fact that there all remaining $p$'s and $w$'s vanish (in \eqref{eqn:Kabasis} and \eqref{eqn:Xibasis}). We now set $\mathcal{E}\coloneqq \{F_{\boldsymbol{\lambda},\boldsymbol{\mu}}\mid F_{\boldsymbol{\lambda},\boldsymbol{\mu}}\in U\}$. Since now $\lambda_a$ and $\mu_i$ can be regarded as independent coordinates on $\mathcal E$, the functions $\{q_a,q_i\}$ are independent when restricted to $\mathcal{E}\subset V_L(\mathcal{I})\cap U$.
\end{proof}

Recall that, for $x_0\in M$, $i(x_0)={\bar x}_0 = i_{23}(i_1(x_0))$. We will also use $i_1(x_0)=(x_0,y_0)$, so that in this notation  $F_{{\bar x}_0}^+=I_{23}((\mathcal{L}^{i_{23}}\overline{A})_{(x_0,y_0)})$. For a generic $(x,y)\in X$, we now define $F_{i_{23}(x,y)}^+\coloneqq I_{23}((\mathcal{L}^{i_{23}}\overline{A})_{(x,y)})$.  Since we already know that $\kappa^{\sigma}|_{F_{i_{23}(x,y)}^+}=0$ for all $(x,y)\in X$ by construction (cf. \eqref{eqn:kappasigma2} and our choice of $\mathcal{K}_a$, $\mathcal{X}_i$, $\mathcal{Y}$ in the construction of the system of PDEs in \eqref{eqn:pde-ck}), it suffices to show that $\beta^c\wedge \bar{\gamma}^y|_{F_{i_{23}(x,y)}^+}=0$ for all $(x,y)\in X$.  
 
\begin{lemma}\label{lem:part3}
The forms $\beta^c\wedge \bar{\gamma}^y$ vanish on $\mathcal{L}^{i_{23}}\overline{A}$ for all $1\leq c\leq q$.
\end{lemma}
\begin{proof}
Recall that
\begin{equation*}
    V_{L+1}(\mathcal{I})\cap U^+=\{E_z^+\in U^+\mid B^c(E_z^+)=K^{\sigma}(E_z^+)=0\}.
\end{equation*}
Since $\{B^c,K^{\sigma}\}$ have linearly independent differentials, $V_{L+1}(\mathcal{I})\cap U^+$ defines a submanifold of $U^+$. Since $\mathcal{I}$ is an ideal, the forms $\beta^c\wedge \bar{k}^a$ and $\beta^c\wedge \bar{\gamma}^{x^i}$ are also in $\mathcal{I}$ and hence also vanish on $V_{L+1}(\mathcal{I})\cap U^+$ (by definition of an integral element). If one sets
\begin{equation*}
    B^{ca}(E_z^+)\coloneqq \beta^c\wedge \bar{k}^a(\kappa_1(E_z^+),\dots,\kappa_{L-p}(E_z^+),\xi_1(E_z^+),\dots,\xi_p(E_z^+),\eta(E_z^+)),
\end{equation*}
\begin{equation*}
    B^{ci}(E_z^+)\coloneqq \beta^c\wedge \bar{\gamma}^{x^i}(\kappa_1(E_z^+),\dots,\kappa_{L-p}(E_z^+),\xi_1(E_z^+),\dots,\xi_p(E_z^+),\eta(E_z^+)),
\end{equation*}
by \Cref{lem:hadamard}, there exist analytic functions $A$ and $M$ on $U^+$ such that the functions $B^{ca}$ and $B^{ci}$ can be expressed as linear combinations of $B^b$ and $K^{\sigma}$:
\begin{equation*}
    B^{ca}=A_{b}^{ca}B^b+M_{\sigma}^{ca}K^{\sigma},\quad B^{ci}=A_{b}^{ci}B^b+M_{\sigma}^{ci}K^{\sigma}.
\end{equation*}
Since $K^{\sigma}(F_{i_{23}(x,y)}^+)=0$ for all $(x,y)\in X$ by construction (since $\kappa^{\sigma}|_{F_{i_{23}(x,y)}^+}=0$), it follows that
\begin{equation*}
    B^{ca}(F_{i_{23}(x,y)}^+)=A_{b}^{ca}(F_{i_{23}(x,y)}^+)B^b(F_{i_{23}(x,y)}^+),\quad B^{ci}(F_{i_{23}(x,y)}^+)=A_{b}^{ci}(F_{i_{23}(x,y)}^+)B^b(F_{i_{23}(x,y)}^+).
\end{equation*}
Similarly, since $\mathcal{I}$ is a differential ideal, $\delta \beta^c\in \mathcal{I}$, and thus if we set
\begin{equation*}
    D^c(E_z^+)\coloneqq \delta \beta^c(\kappa_1(E_z^+),\dots,\kappa_{L-p}(E_z^+),\xi_1(E_z^+),\dots,\xi_p(E_z^+),\eta(E_z^+)),
\end{equation*}
one can find analytic functions $G$ and $H$ on $U^+$ such that
\begin{equation*}
    D^c=G_b^c B^b+H_{\sigma}^c K^{\sigma},
\end{equation*}
and at $F_{i_{23}(x,y)}^+$, we have that
\begin{equation*}
    D^c(F_{i_{23}(x,y)}^+)=G_b^c(F_{i_{23}(x,y)}^+) B^b(F_{i_{23}(x,y)}^+).
\end{equation*}
For notational convenience, we will write
\begin{equation*}
    B^{ca}(x,y)\coloneqq B^{ca}(F_{i_{23}(x,y)}^+),\quad B^{ci}(x,y)\coloneqq B^{ci}(F_{i_{23}(x,y)}^+),\quad D^c(x,y)\coloneqq D^c(F_{i_{23}(x,y)}^+),
\end{equation*}
for the remainder of this proof. For simplicity, we will also index the basis $(\bar{k}^1,\dots,\bar{k}^{L-p},\bar{\gamma}^{x^1},\dots,\bar{\gamma}^{x^p},\bar{\gamma}^y)$ with $(\bar{e}^1,\dots,\bar{e}^{L+1})$, and replace $(B^{ca},B^{ci})$ with $(B^{ca},B^{c(L-p+i)})=(B^{c\alpha})_{\alpha=1}^L$.

Since $\kappa_a(E_z^+),\xi_i(E_z^+),\eta(E_z^+)$ are dual to $\bar{k}^a,\bar{\gamma}^{x^i},\bar{\gamma}^y$, restricting the form $\beta^c$ to $\mathcal{L}^{i_{23}}\overline{A}$, we can expand it as
\begin{equation*}
    \beta^c|_{\mathcal{L}^{i_{23}}\overline{A}}=B^c(x,y)\bar{e}^1\wedge \dots\wedge \bar{e}^L+\sum_{\alpha=1}^{L} (-1)^{L-\alpha+1} B^{c\alpha}(x,y)\bar{e}^1\wedge \dots\wedge \widehat{\bar{e}^{\alpha}}\wedge \dots\wedge \bar{e}^{L+1},
\end{equation*}
where technically ${\bar e}^\alpha$ stands for $I_{23}^*{\bar e}^\alpha$, because of the restriction to ${\mathcal{L}}^{i_{23}}{\overline A}$. 
The anchor map is given by
\begin{equation*}
    \bar{\rho}(\bar{k}_a)=0,\quad \bar{\rho}(\bar{\gamma}_{x^i})=\frac{\partial}{\partial x^i},\quad \bar{\rho}(\bar{\gamma}_y)=\frac{\partial}{\partial y},
\end{equation*}
hence it follows that
\begin{equation*}
    \delta f|_{\mathcal{L}^{i_{23}}\overline{A}}=\frac{\partial f}{\partial x^i}\bar{e}^{L-p+i}+\frac{\partial f}{\partial y}\bar{e}^{L+1}.
\end{equation*}
Furthermore,
\begin{equation*}
    \delta(\bar{e}^1\wedge \dots\wedge \bar{e}^L)|_{\mathcal{L}^{i_{23}}\overline{A}}=(-1)^L L_{\beta,L+1}^{\beta} \bar{e}^1\wedge \dots\wedge \bar{e}^{L+1},
\end{equation*}
\begin{equation*}
    \delta(\bar{e}^1\wedge \dots\wedge \widehat{\bar{e}^{\alpha}}\wedge \dots\wedge \bar{e}^{L+1})|_{\mathcal{L}^{i_{23}}\overline{A}}=(-1)^{\alpha-1} L_{\beta\alpha}^{\beta} \bar{e}^1\wedge \dots\wedge \bar{e}^{L+1}.
\end{equation*}
If one takes the exterior derivative of $\beta^c$, one obtains
\begin{multline*}
    \delta \beta^c|_{\mathcal{L}^{i_{23}}\overline{A}}=\frac{\partial B^c(x,y)}{\partial y}\bar{e}^{L+1}\wedge \bar{e}^1\wedge \dots\wedge \bar{e}^L+B^c(x,y)\delta (\bar{e}^1\wedge \dots\wedge \bar{e}^L)\\+\sum_{\alpha=1}^L (-1)^{L-\alpha+1} \frac{\partial B^{c\alpha}(x,y)}{\partial x^i} \bar{e}^{L-p+i}\wedge e^1\wedge \dots\wedge \widehat{\bar{e}^{\alpha}}\wedge \dots\wedge \bar{e}^L\wedge \bar{e}^{L+1}\\+\sum_{\alpha=1}^L (-1)^{L-\alpha+1} B^{c\alpha}(x,y) \delta(\bar{e}^1\wedge \dots\wedge \widehat{\bar{e}^{\alpha}}\wedge \dots\wedge \bar{e}^{L+1}),
\end{multline*}
where upon simplification, one gets
\begin{align*}
    \delta \beta^c|_{\mathcal{L}^{i_{23}}\overline{A}}&=(-1)^L \left[\frac{\partial B^c}{\partial y}+L_{\beta,L+1}^{\beta}+\frac{\partial B^{c,L-p+i}}{\partial x^i}+B^{c\alpha}L_{\beta \alpha}^{\beta}\right] \bar{e}^1\wedge \dots\wedge \bar{e}^{L+1}\\&=D^c(x,y) \bar{e}^1\wedge \dots\wedge \bar{e}^{L+1}.
\end{align*}
Using that $D^c(x,y)=G_b^c(x,y)B^b(x,y)$ and $B^{c\alpha}(x,y)=A_b^{c\alpha}(x,y)B^b(x,y)$, we see that the functions $B^c(x,y)$ satisfy the system
\begin{equation*}
    \frac{\partial B^c}{\partial y}=\left[(-1)^L G_b^c-L_{\beta,L+1}^{\beta}\delta_b^c-\frac{\partial A_b^{c,L-p+i}}{\partial x^i}-A_b^{c\alpha}L_{\beta \alpha}^{\beta}\right]B^b-A_{b}^{c,L-p+i}\frac{\partial B^b}{\partial x^i}.
\end{equation*}
In fact, it is a linear system of PDEs of the form
\begin{equation*}
    \frac{\partial B^c(x,y)}{\partial y}=\widetilde{G}_b^c(x,y) B^b(x,y)+\widetilde{A}_b^{ci}(x,y) \frac{\partial B^b(x,y)}{\partial x^i},
\end{equation*}
whose initial conditions are given by
\begin{equation*}
    B^c(x,0)=0,
\end{equation*}
where $\widetilde{G}_b^c$ and $\widetilde{A}_b^{ci}$ are analytic functions. By the Cauchy--Kowalevski theorem (cf. \Cref{sec:preliminaries}), it follows that the only solution is $B^c(x,y)=0$. Hence the forms $\beta^c$ vanish on $\mathcal{L}^{i_{23}}\overline{A}$, and hence so do the forms $\beta^c\wedge \bar{\gamma}^y$. \end{proof}
This concludes the proof of \Cref{thm:maintheorem}.
\end{proof}

Applying this theorem inductively gives the following useful corollary (we will also call this the Cartan--K{\"a}hler theorem).

\begin{corollary}[Cartan--K{\"a}hler, second version]\label{cor:maintheorem}
Let $\mathcal{I}\subset \Lambda^{\ast}(\overline{A})$ be an analytic differential ideal of an analytic transitive Lie algebroid $\bar{\tau}\colon \overline{A}\to \overline{M}$, and let $E_z\subset \overline{A}_z$ be a $\dim(\ker(\bar{\rho}_z))$-ordinary integral element of $\mathcal{I}$. Then there exists an integral manifold $j\colon X\to \overline{M}$ of $\mathcal{I}$ which passes through $z$ and satisfies $J((\mathcal{L}^{j}\overline{A})_{x})=E_z$, where $x\in X$ is a point such that $z=j(x)$, and $J$ is the $j$-induced map.
\end{corollary}
Note that when $\overline{A}=\mathrm{T}\overline{M}$, \Cref{cor:maintheorem} reduces to the classical statement of the Cartan--K{\"a}hler theorem which we have also mentioned in the introduction (Chap.\ 3, \S~2, Cor. 2.3 of \cite{bryant1991exterior}): If $E_z\subset \mathrm{T}_z\overline{M}$ is an ordinary integral element, there exists an integral manifold through $z$ whose tangent space at $z$ is precisely $E_z$.
\begin{proof}
This proof uses some ideas from \cite{bryant1991exterior,ivey2016cartan}. For the entirety of this proof, we will use the notation $I_{(\ast)}$ for the $i_{(\ast)}$-induced map. Let $p\coloneqq \dim(E_z)-\dim(\ker(\bar{\rho}_z))$. By hypothesis, there exists an $\dim(\ker(\bar{\rho}_z))$-integral flag
\begin{equation*}
    (E_{\dim(\ker(\bar{\rho}_z))})_z\subset (E_{\dim(\ker(\bar{\rho}_z))+1})_z\subset \dots\subset (E_{\dim(\ker(\bar{\rho}_z))+p})_z=E_z,
\end{equation*}
such that $\dim (E_{\dim(\ker(\bar{\rho}_z))+q})_z=\dim(\ker(\bar{\rho}_z))+q$ for $0\leq q\leq p$ and $(E_{\dim(\ker(\bar{\rho}_z))+q})_z$ is K{\"a}hler-regular for $0\leq q\leq p-1$. Define
\begin{equation*}
    c_q\coloneqq \codim H((E_{\dim(\ker(\bar{\rho}_z))+q})_z)=\dim \overline{A}_z-\dim H((E_{\dim(\ker(\bar{\rho}_z))+q})_z),
\end{equation*}
for $0\leq q\leq p-1$ and note that the $c_q$ are non-decreasing. Let $\bar{\gamma}\colon \mathrm{T}\overline{M}\to \overline{A}$ be a splitting of $\bar{\rho}$, and let $x^1,\dots,x^p,y^1,\dots,y^s$ (where $s\coloneqq \dim(\overline{M})-p$) be local coordinates of $\overline{M}$ centered at $z$ defined on an open set $V\subset \overline{M}$ chosen such that
\begin{equation*}
    (E_{\dim(\ker(\bar{\rho}_z))+q})_z=\spn\{\bar{k}_a(z),\bar{\gamma}_{x^1}(z),\dots,\bar{\gamma}_{x^q}(z)\}
\end{equation*}
for $0\leq q\leq p$ and $H((E_{\dim(\ker(\bar{\rho}_z))+q})_z)$ is annihilated by $\bar{\gamma}^{y^1}(z),\dots,\bar{\gamma}^{y^{c_q}}(z)$ for $0\leq q\leq p-1$. Here, $\bar{\gamma}_{x^i}\coloneqq \bar{\gamma}\left(\frac{\partial}{\partial x^i}\right)$, $\bar{\gamma}_{y^k}\coloneqq \bar{\gamma}\left(\frac{\partial}{\partial y^k}\right)$, and $\{\bar{k}^a,\bar{\gamma}^{x^i},\bar{\gamma}^{y^k}\}$ denotes the dual basis of $\{\bar{k}_a,\bar{\gamma}_{x^i},\bar{\gamma}_{y^k}\}$.

Now, let $R_0\subset V$ be the submanifold defined by setting $x^2=\dots=x^p=0$ and $y^k=f_1^k(x^1)$ for $k>c_0$, where $f_1^k$ are some analytic functions such that $f_1^k(0)=0$ and $\frac{\partial f_1^k}{\partial x^1}(0)=0$. We now wish to show the existence of a $1$-dimensional integral manifold $X_1\subset R_0$ using the Cartan--K{\"a}hler theorem. To do so, we require that $R_0$ has codimension $r((E_{\dim(\ker(\bar{\rho}_z))})_z)$. This is indeed the case, since the number of unconstrained coordinates is
\begin{align*}
    1+c_0=1+\dim \overline{A}_z-\dim H((E_{\dim(\ker(\bar{\rho}_z))})_z)&=1+N-(r((E_{\dim(\ker(\bar{\rho}_z))})_z)+\dim(\ker(\bar{\rho}_z))+1)\\&=1+N-(r((E_{\dim(\ker(\bar{\rho}_z))})_z)+N-\dim \overline{M}+1)\\&=\dim \overline{M}-r((E_{\dim(\ker(\bar{\rho}_z))})_z).
\end{align*}
Let $X_0=\{x_0\}$ be a single point and let $i_0\colon X_0\to \overline{M}$ be the map $x_0\mapsto z$. Let $i_{12,0}\colon X_0\to R_0$ and $i_{3,0}\colon R_0\to \overline{M}$ be the inclusions. Then $\mathrm{T}_{i_{12,0}(x_0)}R_0$ is spanned by the vectors
\begin{equation*}
    \left\{\frac{\partial}{\partial x^1}\Big|_{i_{12,0}(x_0)},\frac{\partial}{\partial y^1}\Big|_{i_{12,0}(x_0)},\dots,\frac{\partial}{\partial y^{c_0}}\Big|_{i_{12,0}(x_0)}\right\}.
\end{equation*}
Using that $\mathcal{L}^{i_{3,0}}\overline{A}\simeq \ker(\tilde{\rho})\oplus \tilde{\gamma}(\mathrm{T} R)$, where $\tilde{\rho}\colon \mathcal{L}^{i_{3,0}}\overline{A}\to \mathrm{T}R_0$ is the anchor map of the $i_{3,0}$-prolongation of $\overline{A}$, and $\tilde{\gamma}\colon \mathrm{T}R_0\to \mathcal{L}^{i_{3,0}}\overline{A}$ is the $(i_{3,0},\bar{\gamma})$-induced splitting, one sees that
\begin{equation*}
    I_{3,0}((\mathcal{L}^{i_{3,0}}\overline{A})_{i_{12,0}(x_0)})=\spn\left\{\bar{k}_a(z),\bar{\gamma}_{x^1}(z),\bar{\gamma}_{y^1}(z),\dots,\bar{\gamma}_{y^{c_0}}(z)\right\}.
\end{equation*}
Finally, since $\spn\left\{\bar{\gamma}_{y^1}(z),\dots,\bar{\gamma}_{y^{c_0}}(z)\right\}$ is complementary to $H((E_{\dim(\ker(\bar{\rho}_z))})_z)$, it follows that
\begin{equation*}
    \overline{A}_z=\spn\left\{\bar{\gamma}_{y^1}(z),\dots,\bar{\gamma}_{y^{c_0}}(z)\right\}\oplus H((E_{\dim(\ker(\bar{\rho}_z))})_z),
\end{equation*}
therefore $\overline{A}_z=I_{3,0}((\mathcal{L}^{i_{3,0}}\overline{A})_{i_{12,0}(x_0)})+H((E_{\dim(\ker(\bar{\rho}_z))})_z)$, since
\begin{equation*}
    I_{3,0}((\mathcal{L}^{i_{3,0}}\overline{A})_{i_{12,0}(x_0)})\supset \spn\{\bar{\gamma}_{y^1}(z),\dots,\bar{\gamma}_{y^{c_0}}(z)\}.
\end{equation*}
It remains to show that $(E_{\dim(\ker(\bar{\rho}_z))})_z=I_0((\mathcal{L}^{i_0}\overline{A})_{x_0})$. This is indeed the case, because
\begin{equation*}
    \mathcal{L}^{i_0}\overline{A}=\{(\bar{a}_z,0)\in \overline{A}\times \mathrm{T}X_0\mid \bar{\rho}(\bar{a}_z)=0\}\simeq \ker(\bar{\rho}).
\end{equation*}
Denote the inclusion $X_i\to X_l$ for $i\leq l$ by $j_{i,l}$. Hence, by using the Cartan--K{\"a}hler theorem, we know that there exists a $1$-dimensional, connected, analytic integral manifold $i_1\colon X_1\to \overline{M}$ of $\mathcal{I}$ satisfying $X_0\subset X_1\subset R_0$. Besides, the construction is such that $I_1((\mathcal{L}^{i_1}\overline{A})_{j_{0,1}(x_0)})=(E_{\dim(\ker(\bar{\rho}_z))+1})_z$. Indeed, if $X_1\subset R_0$, then $I_1((\mathcal{L}^{i_1}\overline{A})_{j_{0,1}(x_0)})\subset I_{3,0}((\mathcal{L}^{i_{3,0}}\overline{A})_{i_{12,0}(x_0)})$. Assuming $X_1$ is a $1$-dimensional integral manifold, it follows from \Cref{prop:intmfdintelement} and \Cref{prop:polar} that $I_1((\mathcal{L}^{i_1}\overline{A})_{j_{0,1}(x_0)})\subset H((E_{\dim(\ker(\bar{\rho}_z))})_z)$ (note that $\dim I_1((\mathcal{L}^{i_1}\overline{A})_{j_{0,1}(x_0)})=\dim I_0((\mathcal{L}^{i_0}\overline{A})_{x_0})+1$). Therefore, it follows that $I_1((\mathcal{L}^{i_1}\overline{A})_{j_{0,1}(x_0)})\subset I_{3,0}((\mathcal{L}^{i_{3,0}}\overline{A})_{i_{12,0}(x_0)})\cap H((E_{\dim(\ker(\bar{\rho}_z))})_z)=\spn\{\bar{k}_a(z),\bar{\gamma}_{x^1}(z)\}=(E_{\dim(\ker(\bar{\rho}_z))+1})_z$, but since both sides are $1$-dimensional, it follows that $I_1((\mathcal{L}^{i_1}\overline{A})_{j_{0,1}(x_0)})=(E_{\dim(\ker(\bar{\rho}_z))+1})_z$.

Next, let $R_1\subset V$ be the submanifold of codimension $r((E_{\dim(\ker(\bar{\rho}_z))+1})_z)$ defined by setting $x^3=\dots=x^p=0$ and $y^k=f_2^k(x^1,x^2)$ for $k>c_1$, where $f_2^k$ are some real analytic functions such that $f_2^k(x^1,0)=f_1^k(x^1)$ (so that $R_1$ contains $X_1$) and $\frac{\partial f_2^k}{\partial x^1}(0,0)=\frac{\partial f_2^k}{\partial x^2}(0,0)=0$ (note that the condition $\frac{\partial f_2^k}{\partial x^1}(0,0)=0$ is consistent with the condition $\frac{\partial f_1^k}{\partial x^1}(0)=0$ on $f_1^k$). Then, by a similar argument, if $i_{12,0}\colon X_1\to R_1$ and $i_{3,1}\colon R_1\to \overline{M}$ we can see that
\begin{equation*}
    I_{3,1}((\mathcal{L}^{i_{3,1}}\overline{A})_{i_{12,1}(j_{0,1}(x_0))})+H(I_1((\mathcal{L}^{i_1}\overline{A})_{j_{0,1}(x_0)}))=\overline{A}_{z}.
\end{equation*}
Indeed,
\begin{equation*}
    I_{3,1}((\mathcal{L}^{i_{3,1}}\overline{A})_{i_{12,1}(j_{0,1}(x_0))})=\spn\left\{\bar{k}_a(z),\bar{\gamma}_{x^1}(z),\bar{\gamma}_{y^1}(z),\dots,\bar{\gamma}_{y^{c_1}}(z)\right\}.
\end{equation*}
We can then shrink $V$ so that
\begin{equation*}
    I_{3,1}((\mathcal{L}^{i_{3,1}}\overline{A})_{i_{12,1}(y)})+H(I_1((\mathcal{L}^{i_1}\overline{A})_y))=\overline{A}_{i_1(y)},
\end{equation*}
for all $y\in X_1$. By a similar reasoning as before, using the Cartan--K{\"a}hler theorem, one obtains a $2$-dimensional integral manifold $i_2\colon X_2\to \overline{M}$ of $\mathcal{I}$ (with $X_1\subset X_2\subset R_1$) such that $I_2((\mathcal{L}^{i_2}\overline{A})_{j_{0,2}(x_0)})=(E_{\dim(\ker(\bar{\rho}_z))+2})_z$.

Continuing in this manner, we obtain a $p$-dimensional integral manifold $i_p\colon X_p\to \overline{M}$ of $\mathcal{I}$ (with $X_{p-1}\subset X_p\subset R_{p-1}$) such that $I_p((\mathcal{L}^{i_p}\overline{A})_{j_{0,p}(x_0)})=E_z$ (in the last step, $R_{p-1}\subset V$ is defined by $y^k=f_p^k(x^1,\dots,x^p)$ for $k>c_{p-1}$). It follows that $j=i_p$ is the desired integral manifold.

In summary, we have used the following notation  (we omit $i_l\colon X_l\to \overline{M}$ for $1\leq l\leq p-1$ from the diagram for clarity):
\begin{equation*}
\begin{tikzcd}
	{X_0} & {R_0} && \\
	{X_1} & {R_1} \\
	{X_2} & {R_2} && {\overline{M}} \\
	\vdots & \vdots \\
	{X_{p-1}} & {R_{p-1}} \\
	{X_p}
	\arrow["{i_{12,0}}", from=1-1, to=1-2]
	\arrow["{j_{0,1}}", from=1-1, to=2-1]
    \arrow["{j_{0,2}}"', bend left=-25, from=1-1, to=3-1]
    \arrow["{i_0}", bend left=70, from=1-1, to=3-4]
	\arrow["{j_{0,p-1}}", bend left=-50, from=1-1, to=5-1]
	\arrow["{j_{0,p}}"', bend left=-55, from=1-1, to=6-1]
	\arrow[from=1-2, to=2-2]
	\arrow["{i_{3,0}}", bend left=10, from=1-2, to=3-4]
	\arrow["{i_{12,1}}", from=2-1, to=2-2]
	\arrow["{j_{1,2}}", from=2-1, to=3-1]
	\arrow[from=2-2, to=3-2]
	\arrow["{i_{3,1}}", from=2-2, to=3-4]
	\arrow["{i_{12,2}}", from=3-1, to=3-2]
	\arrow[from=3-1, to=4-1]
	\arrow["{i_{3,2}}", from=3-2, to=3-4]
	\arrow[from=3-2, to=4-2]
	\arrow[from=4-1, to=5-1]
	\arrow[from=4-2, to=5-2]
	\arrow["{i_{12,p-1}}", from=5-1, to=5-2]
	\arrow["{j_{p-1,p}}", from=5-1, to=6-1]
	\arrow["{i_{3,p-1}}", from=5-2, to=3-4]
	\arrow["{j=i_p}"', bend left=-30, from=6-1, to=3-4]
\end{tikzcd}
\end{equation*}
The inclusions $j_{l,l+1}\colon X_l\to X_{l+1}$ for $0\leq l\leq p-1$ correspond to the extension of $X_l$ to $X_{l+1}$ by a single application of the Cartan--K{\"a}hler theorem.
\end{proof}

%% file: sections/applications.tex
In this section, we provide an example where we apply the Cartan--K{\"a}hler theorem. In particular, for a specific EDS on a Lie algebroid, we will construct a family of integral elements which we will denote by $\{E_z\mid z\in \overline{M}\}$. We will show that for each element in this family, there exists an integral manifold with the necessary properties in \Cref{cor:maintheorem}.

Let $\overline{A}=\mathrm{T}\mathbb{R}^3\times \mathbb{R}^1$ be a vector bundle over $\overline{M}=\mathbb{R}^3$. If $(x^1,x^2,x^3)$ are coordinates on $\overline{M}$, then a basis of sections of $\overline{A}$ is given by $\{\frac{\partial}{\partial x^1},\frac{\partial}{\partial x^2},\frac{\partial}{\partial x^3},\bar{k}_1\}$, where $\bar{k}_1\colon \overline{M}\to \overline{A}$ is given by $z\mapsto (0,0,0,1)$. We denote its dual basis by $\{\mathrm{d}x^1,\mathrm{d}x^2,\mathrm{d}x^3,\bar{k}^1\}$. We define a Lie algebroid structure on $\overline{A}$ by defining the anchor $\bar{\rho}\colon \overline{A}\to \mathrm{T}\mathbb{R}^3$ on the given basis of sections such that
\begin{equation*}
    \bar{\rho}\left(\frac{\partial}{\partial x^i}\right)=\frac{\partial}{\partial x^i},\quad \bar{\rho}(\bar{k}_1)=0,
\end{equation*}
and its Lie bracket such that
\begin{equation*}
    \left\llbracket\frac{\partial}{\partial x^i},\frac{\partial}{\partial x^j}\right\rrbracket=\left\llbracket\frac{\partial}{\partial x^i},\bar{k}_1\right\rrbracket=0.
\end{equation*}
Let $\mathcal{I}=\langle \theta^1=\mathrm{d}x^3-x^2\, \mathrm{d}x^1,\theta^2=\mathrm{d}x^2\rangle_{\mathrm{alg}}$. Since
\begin{equation*}
    \delta(\mathrm{d}x^i)=0,\quad \delta x^2=\mathrm{d}x^2,
\end{equation*}
it follows that
\begin{equation*}
    \delta \theta^1=-\mathrm{d}x^2\wedge \mathrm{d}x^1=\mathrm{d}x^1\wedge \mathrm{d}x^2=\mathrm{d}x\wedge \theta^2\in \mathcal{I},
\end{equation*}
\begin{equation*}
    \delta \theta^2=0\in \mathcal{I}.
\end{equation*}
Hence, $\mathcal{I}$ is a differential ideal. At each $z\in \overline{M}$, there is a unique integral element of dimension $2$, namely
\begin{equation*}
    E_z\coloneqq \{v\in \overline{A}_z\mid \theta^a(v)=0~\mathrm{for}~1\leq a\leq 2\}=\spn\left\{\left(\frac{\partial}{\partial x^1}+x^2 \frac{\partial}{\partial x^3}\right)\Big|_z,\bar{k}_1(z)\right\}.
\end{equation*}
In fact, every integral element of $\mathcal{I}$ based at $z$ must be a subspace of $E_z$, since $H((0)_z)=E_z$. Thus, if
\begin{equation*}
    (E_1)_z\subset (E_2)_z=E_z
\end{equation*}
is any $1$-integral flag, then $H((E_1)_z)=E_z$. We must now show that $(E_2)_z=E_z$ is a $1$-ordinary integral element. We therefore only need to find a $1$-dimensional K{\"a}hler-regular subspace $(E_1)_z$ of $E_z$.

We claim that $(E_1)_z\coloneqq \spn\{\bar{k}_1(z)\}$ is a K{\"a}hler-regular integral element. Since $(E_1)_z\subset E_z$, it follows from \Cref{prop:subspaceintelement} that $(E_1)_z$ is an integral element. Now, consider $\Omega=\bar{k}^1\in \Lambda^1(\overline{A})$. Then $\Omega_{(E_1)_z}\neq 0$, and we have that
\begin{equation*}
    \mathcal{F}_{\Omega}(\mathcal{I})=\{\theta_{\Omega}\mid \theta\in \mathcal{I}^1\}=\spn\{\theta_{\Omega}^1,\theta_{\Omega}^2\}.
\end{equation*}
Every $(F_1)_{\bar{m}}\in G_1(\overline{A},\Omega)$ has a unique basis of the form
\begin{equation*}
    X_1((F_1)_{\bar{m}})=\bar{k}_1(\bar{m})+p^1((F_1)_{\bar{m}}) \frac{\partial}{\partial x^1}\Big|_{\bar{m}}+p^2((F_1)_{\bar{m}}) \frac{\partial}{\partial x^2}\Big|_{\bar{m}}+p^3((F_1)_{\bar{m}}) \frac{\partial}{\partial x^3}\Big|_{\bar{m}}.
\end{equation*}
Therefore, the functions $x^1,x^2,x^3,p^1,p^2,p^3$ form a coordinate system on $G_1(\overline{A},\Omega)$. Note that
\begin{equation*}
    \theta_{\Omega}^1((F_1)_{\bar{m}})=\frac{\theta^1(X_1)}{\Omega(X_1)}=-x^2 p^1+p^3,\quad \theta_{\Omega}^2((F_1)_{\bar{m}})=p^2.
\end{equation*}
Note that the common zeros of $\theta_{\Omega}^1$ and $\theta_{\Omega}^2$ are precisely $V_1(\mathcal{I},\Omega)$. We also have that
\begin{equation*}
    Z(\mathcal{F}_{\Omega}(\mathcal{I}))=\{(F_1)_{\bar{m}}\in G_1(\overline{A},\Omega)\mid \theta_{\Omega}^1((F_1)_{\bar{m}})=\theta_{\Omega}^2((F_1)_{\bar{m}})=0\},
\end{equation*}
and moreover, $\theta_{\Omega}^1, \theta_{\Omega}^2$ have linearly independent differentials, since
\begin{equation*}
    \mathrm{d}\theta_{\Omega}^1=-p^1\, \mathrm{d}x^2-x^2\, \mathrm{d}p^1+\mathrm{d}p^3,\quad \mathrm{d}\theta_{\Omega}^2=\mathrm{d}p^2.
\end{equation*}
It follows that $(E_1)_z$ is K{\"a}hler-ordinary (in fact, our reasoning shows that every point in $V_1(\mathcal{I},\Omega)$ is K{\"a}hler-ordinary). It is also K{\"a}hler-regular, since $H((F_1)_{\bar{m}})=E_{\bar{m}}$. Since $\theta^1,\theta^2$ are pointwise linearly independent, it follows that $\dim(H((F_1)_{\bar{m}}))=2$ for all $\bar{m}\in \overline{M}$, and therefore $r$ is a locally constant function. We may therefore conclude using the second version of the Cartan--K{\"a}hler theorem that for every $z\in \overline{M}$, there exists an integral manifold $j\colon X\to \overline{M}$ of $\mathcal{I}$ which passes through $z$ and satisfies $J((\mathcal{L}^j \overline{A})_x)=E_z$, where $x\in X$ is a point such that $z=j(x)$.

Although the Cartan--K{\"a}hler theorem only guarantees existence, in this case it is possible to give an explicit example of such an integral manifold for each $z=(x_0^1,x_0^2,x_0^3)\in \overline{M}$. We claim that
\begin{equation*}
    j\colon \mathbb{R}\longrightarrow \overline{M},\quad x^1\longmapsto \left(x^1,x_0^2,x_0^2(x^1-x_0^1)+x_0^3\right),
\end{equation*}
defines such an integral manifold. Note that if $\bar{a}_{\bar{m}}=a_1\frac{\partial}{\partial x^1}+a_2 \frac{\partial}{\partial x^2}+a_3 \frac{\partial}{\partial x^3}+a_4 \bar{k}_1$ and $v_m=b_1 \frac{\partial}{\partial x^1}$, then $\bar{\rho}(\bar{a}_{\bar{m}})=\mathrm{T}j(v_m)$ requires $a_1=b_1$ and $a_3=x_0^2 b_1$. Hence, one has
\begin{equation*}
    \mathcal{L}^j \overline{A}=\spn\left\{\left(\frac{\partial}{\partial x^1}+x_0^2 \frac{\partial}{\partial x^3},\frac{\partial}{\partial x^1}\right),\left(\bar{k}_1,0\right) \right\}.
\end{equation*}
Applying $J$ gives
\begin{equation*}
    J((\mathcal{L}^j \overline{A})_{x_0^1})=\spn\left\{\frac{\partial}{\partial x^1}\Big|_{j(x_0^1)}+x_0^2 \frac{\partial}{\partial x^3}\Big|_{j(x_0^1)},\bar{k}_1(j(x_0^1)) \right\}=E_z,
\end{equation*}
as required.

%% file: sections/inv-inverse.tex
This section is based on the discussion of the EDS approach to the time-dependent invariant inverse problem in \cite{mestdag2026edsla}. In that article, we showed that we may solve the following problem by finding integral manifolds of an appropriate EDS on a Lie algebroid.

\emph{Given a (possibly time-dependent) $G$-invariant second-order system on a Lie group $G$, when does a (possibly time-dependent) regular $G$-invariant Lagrangian $L$ whose Euler--Lagrange equations are equivalent to this system exist?}

The appropriate Lie algebroid is constructed as follows. Let $\mathfrak{g}$ be the Lie algebra of $G$. Consider the vector bundle
\begin{equation*}
    A=\mathbb{R}\times \mathfrak{g}\times \mathbb{R}\times \mathfrak{g}\times \mathfrak{g}\longrightarrow M=\mathbb{R}\times \mathfrak{g}.
\end{equation*}
We put coordinates $(w^i)$ on $\mathfrak{g}$. A basis of sections of this bundle is given by $\{T_0,e_i,W_i\}$, where
\begin{equation*}
    T_0\colon (t,w)\longmapsto (t,w,1,0,0),
\end{equation*}
\begin{equation*}
    e_i\colon (t,w)\longmapsto (t,w,0,E_i,0),\quad W_i\colon (t,w)\longmapsto (t,w,0,0,E_i),
\end{equation*}
where $E_i$ is a basis of $\mathfrak{g}$. If $C_{ij}^k$ are the structure constants of the Lie algebra, the anchor map satisfies
\begin{equation*}
    \rho(T_0)=\frac{\partial}{\partial t},\quad \rho(e_i)=w^k C_{ki}^j \frac{\partial}{\partial w^j},\quad \rho(W_i)=\frac{\partial}{\partial w^i},
\end{equation*}
and its bracket is given by
\begin{equation*}
    \llbracket T_0,e_i\rrbracket=0,\quad \llbracket T_0,W_i\rrbracket=0,
\end{equation*}
\begin{equation*}
    \llbracket e_i,e_j\rrbracket=C_{ij}^k e_k,\quad \llbracket e_i,W_j\rrbracket=C_{ij}^k W_k,\quad \llbracket W_i,W_j\rrbracket=0.
\end{equation*}
We call $A$ the \emph{IP Lie algebroid}. We then carried out a convenient change of basis, yielding the basis $\{\Gamma_0,W_i,H_i\}$, whose dual we denoted by $\{\Gamma^0,\Psi^i,\Theta^i\}$. We then showed (cf. Thm. 2 of \cite{mestdag2026edsla}) that a reduced multiplier matrix exists for a given $G$-invariant second order system if and only if integral manifolds of a special type exist for an EDS generated by a set of $1$-forms on a specific $p$-prolongation Lie algebroid of $A$. The ``special type'' in the previous sentence refers to a discussion on an \emph{independence condition}, that we will come back to at the end of this section.

For the sake of simplicity, we will consider the invariant inverse problem for the canonical connection on the $1$-dimensional Lie group $G=\mathbb{R}^1$ of the real line (as considered in Example 4 of \cite{mestdag2026edsla}). In this case, the $p$-prolongation Lie algebroid, which we will denote by $\overline{A}\coloneqq \mathcal{L}^p A$ is given by
\begin{equation*}
    \overline{A}=A\times \mathrm{T}\mathbb{R}^1\times \mathrm{T}\mathbb{R}^2\longrightarrow \overline{M}=M\times \mathbb{R}^1\times \mathbb{R}^2,
\end{equation*}
where $\tau\colon A\to M$ has a basis of sections $\{\Gamma_0,W_1,H_1\}$. If $(t,w,s,P,Q)$ are the coordinates of $\overline{M}$ (note that in \cite{mestdag2026edsla}, we used the notation $(t,w^1,s_{11},P_{111},Q_{111})$ instead), the given basis of $A$ extends to a basis of $\overline{A}$, given by $\{\Gamma_0,W_1,H_1,\frac{\partial}{\partial s},\frac{\partial}{\partial P},\frac{\partial}{\partial Q}\}$. Its dual basis is given by $\{\Gamma^0,\Psi^1,\Theta^1,\delta s,\delta P,\delta Q\}$. The suitable differential ideal for this problem is given by $\mathcal{I}=\langle \sigma_{11}\rangle=\langle \sigma_{11},\delta \sigma_{11} \rangle_{\mathrm{alg}}$, where
\begin{equation*}
    \sigma_{11}=\delta s+P \Psi^1+Q \Theta^1,
\end{equation*}
\begin{equation*}
    \delta \sigma_{11}=(\delta P+Q \Gamma^0)\wedge \Psi^1+\delta Q\wedge \Theta^1.
\end{equation*}
There, it was shown that integral manifolds of $\mathcal{I}$ are generally given by
\begin{equation*}
    i\colon M\longrightarrow \overline{M},\quad (t,w)\longmapsto \left(t,w,s=g(w),P=-\frac{\partial g(w)}{\partial w},Q=0\right),
\end{equation*}
where $g$ is any nowhere zero smooth function.

We now use the Cartan--K{\"a}hler theorem to construct integral manifolds of this type. Let $z\coloneqq (0,0,c_1,c_2,0)$, where $c_1\neq 0$ and $c_2\in \mathbb{R}$. Let $d_1,d_2,d_3\in \mathbb{R}$ and consider
\begin{equation*}
    E_z\coloneqq \spn\left\{\Gamma_0(z),W_1(z)+d_1 \frac{\partial}{\partial s}\Big|_z+d_2 \frac{\partial}{\partial P}\Big|_z+d_3\frac{\partial}{\partial Q}\Big|_z,H_1(z)\right\}.
\end{equation*}
 Note that $E_z$ is an integral element based at $z$ if and only if $d_1=-c_2$ and $d_3=0$, because
\begin{multline*}
    (\beta\wedge \sigma_{11})_z\left(\Gamma_0(z),W_1(z)+d_1 \frac{\partial}{\partial s}\Big|_z+d_2 \frac{\partial}{\partial P}\Big|_z+d_3\frac{\partial}{\partial Q}\Big|_z,H_1(z)\right)\\=\beta_z\left(\Gamma_0(z),W_1(z)+d_1 \frac{\partial}{\partial s}\Big|_z+d_2 \frac{\partial}{\partial P}\Big|_z+d_3\frac{\partial}{\partial Q}\Big|_z\right)Q(z)-\beta_z(\Gamma_0(z),H_1(z))(P(z)+d_1)\\=-\beta_z(\Gamma_0(z),H_1(z))(c_2+d_1),
\end{multline*}
and
\begin{multline*}
    (\alpha\wedge \delta \sigma_{11})_z\left(\Gamma_0(z),W_1(z)+d_1 \frac{\partial}{\partial s}\Big|_z+d_2 \frac{\partial}{\partial P}\Big|_z+d_3\frac{\partial}{\partial Q}\Big|_z,H_1(z)\right)\\=\alpha_z(\Gamma_0(z)) d_3+\alpha_z(H_1(z))Q(z)=\alpha_z(\Gamma_0(z)) d_3,
\end{multline*}
for every $\alpha\in \Lambda^1(\overline{A})$ and $\beta\in \Lambda^2(\overline{A})$. From now on, we assume that  $d_1=-c_2$ and $d_3=0$, so that
\begin{equation*}
    E_z\coloneqq \spn\left\{\Gamma_0(z),W_1(z)-c_2 \frac{\partial}{\partial s}\Big|_z+d_2 \frac{\partial}{\partial P}\Big|_z,H_1(z)\right\}
\end{equation*}
is an integral element. We claim that there exists an integral manifold $j\colon X\to \overline{M}$ of $\mathcal{I}$ such that $J((\mathcal{L}^j \overline{A})_x)=E_z$ for all $d_2\in \mathbb{R}$, where $x\in X$ is a point such that $z=j(x)$.

Let
\begin{equation*}
    (E_1)_z\coloneqq \spn\{H_1(z)\},\quad (E_2)_z\coloneqq \spn\left\{\Gamma_0(z),H_1(z)+W_1(z)-c_2 \frac{\partial}{\partial s}\Big|_z+d_2\frac{\partial}{\partial P}\Big|_z\right\}.
\end{equation*}
We will show that $(E_1)_z\subset (E_2)_z\subset E_z$ is a $1$-ordinary integral flag. It follows from \Cref{prop:subspaceintelement} that $(E_1)_z$ and $(E_2)_z$ are integral elements. We now show that $(E_1)_z$ is K{\"a}hler-ordinary. Let $\Omega=\Theta^1$. Then $\Omega|_{(E_1)_z}\neq 0$ because $\Omega(H_1)=1$. We have that
\begin{equation*}
    \mathcal{F}_{\Omega}(\mathcal{I})=\spn\{(\sigma_{11})_{\Omega}\}.
\end{equation*}
Every $(F_1)_{\bar{m}}\in G_1(\overline{A},\Omega)$ has a unique basis of the form
\begin{multline*}
    X_1((F_1)_{\bar{m}})=H_1(\bar{m})+p^1((F_1)_{\bar{m}}) \Gamma_0(\bar{m})+p^2((F_1)_{\bar{m}}) W_1(\bar{m})\\+p^4((F_1)_{\bar{m}}) \frac{\partial}{\partial s}\Big|_{\bar{m}}+p^5((F_1)_{\bar{m}}) \frac{\partial}{\partial P}\Big|_{\bar{m}}+p^6((F_1)_{\bar{m}}) \frac{\partial}{\partial Q}\Big|_{\bar{m}}.
\end{multline*}
A coordinate system on $G_1(\overline{A},\Omega)$ is hence given by $t,w,s,P,Q,p^1,p^2,p^4,p^5,p^6$. Note that
\begin{equation*}
    (\sigma_{11})_{\Omega}((F_1)_{\bar{m}})=\frac{\sigma_{11}(X_1)}{\Omega(X_1)}=p^4+P p^2+Q.
\end{equation*}
Note that
\begin{equation*}
    Z(\mathcal{F}_{\Omega}(\mathcal{I}))=V_1(\mathcal{I},\Omega)=\{(F_1)_{\bar{m}}\in G_1(\overline{A},\Omega)\mid (\sigma_{11})_{\Omega}((F_1)_{\bar{m}})=0\},
\end{equation*}
and therefore $(E_1)_z$ is K{\"a}hler-ordinary.

We now show that $(E_2)_z$ is K{\"a}hler-ordinary. Then $\Omega=\Gamma^0\wedge \Theta^1\in \Lambda^2(\overline{A})$ is such that $\Omega|_{(E_2)_z}\neq 0$. We have that
\begin{equation*}
    \mathcal{F}_{\Omega}(\mathcal{I})=\{(\alpha\wedge \sigma_{11})_{\Omega}+(b\, \delta \sigma_{11})_{\Omega}\mid \alpha\in \Lambda^1(\overline{A}),b\in C^{\infty}(\overline{M})\}.
\end{equation*}
Every $(F_2)_{\bar{m}}\in G_2(\overline{A},\Omega)$ has a unique basis of the form
\begin{multline*}
    Y_1((F_2)_{\bar{m}})=\Gamma_0(\bar{m})+p_1^2((F_2)_{\bar{m}}) W_1(\bar{m})+p_1^4((F_2)_{\bar{m}}) \frac{\partial}{\partial s}\Big|_{\bar{m}}\\+p_1^5((F_2)_{\bar{m}}) \frac{\partial}{\partial P}\Big|_{\bar{m}}+p_1^6((F_2)_{\bar{m}}) \frac{\partial}{\partial Q}\Big|_{\bar{m}},
\end{multline*}
\begin{multline*}
    Y_2((F_2)_{\bar{m}})=H_1(\bar{m})+p_2^2((F_2)_{\bar{m}}) W_1(\bar{m})+p_2^4((F_2)_{\bar{m}}) \frac{\partial}{\partial s}\Big|_{\bar{m}}\\+p_2^5((F_2)_{\bar{m}}) \frac{\partial}{\partial P}\Big|_{\bar{m}}+p_2^6((F_2)_{\bar{m}}) \frac{\partial}{\partial Q}\Big|_{\bar{m}}.
\end{multline*}
The functions $t,w,s,P,Q,p_i^j$ form a coordinate system on $G_2(\overline{A},\Omega)$. Note that
\begin{align*}
    (\alpha\wedge \sigma_{11})_{\Omega}((F_2)_{\bar{m}})=\frac{(\alpha\wedge \sigma_{11})(Y_1,Y_2)}{\Omega(Y_1,Y_2)}&=\alpha(Y_1)\sigma_{11}(Y_2)-\alpha(Y_2)\sigma_{11}(Y_1)\\&=\alpha(Y_1)(p_2^4+P p_2^2)-\alpha(Y_2)(p_1^4+P p_1^2+Q).
\end{align*}
and
\begin{equation*}
    (\delta \sigma_{11})_{\Omega}((F_2)_{\bar{m}})=p_2^2(p_1^5+Q)-p_2^5 p_1^2+p_1^6.
\end{equation*}
It follows that
\begin{equation*}
    Z(\mathcal{F}_{\Omega}(\mathcal{I}))=\{(F_2)_{\bar{m}}\in G_2(\overline{A},\Omega)\mid p_2^4+P p_2^2=p_1^4+P p_1^2+Q=p_2^2(p_1^5+Q)-p_2^5 p_1^2+p_1^6=0\}.
\end{equation*}
The differentials of $p_2^4+P p_2^2$, $p_1^4+P p_1^2+Q$ and $p_2^2(p_1^5+Q)-p_2^5 p_1^2+p_1^6$ are linearly independent, hence $(E_2)_z$ is K{\"a}hler-ordinary.

We now show that both $(E_1)_z$ and $(E_2)_z$ are K{\"a}hler-regular. Let $(F_1)_{\bar{m}}=\langle X_1\rangle$ be any $1$-dimensional integral element. Then
\begin{equation*}
    H(\langle X_1\rangle)=\left\{v\in \overline{A}_{\bar{m}}\mid \alpha \wedge \sigma_{11}(v,X_1)=0, \delta \sigma_{11}(v,X_1)=0~\mathrm{for}~\alpha\in \Lambda^1(\overline{A})\right\}.
\end{equation*}
Note that
\begin{equation*}
    (\alpha\wedge \sigma_{11})(v,X_1)=\alpha(v)\sigma_{11}(X_1)-\alpha(X_1)\sigma_{11}(v)=-\alpha(X_1)\sigma_{11}(v),
\end{equation*}
which is zero for all $\alpha\in \Lambda^1(\overline{A})$ if and only if $\sigma_{11}(v)=0$. If
\begin{equation*}
    v=\left(a_1 \Gamma_0+a_2 W_1+a_3 H_1+a_4 \frac{\partial}{\partial s}+a_5 \frac{\partial}{\partial P}+a_6 \frac{\partial}{\partial Q}\right)\Big|_{\bar{m}},
\end{equation*}
then this equation is given by
\begin{equation*}
    a_4+P a_2+Q a_3=0.
\end{equation*}
Note that
\begin{align*}
    \delta \sigma_{11}(v,X_1)&=\delta \sigma_{11}\left(v,H_1+p^1 \Gamma_0+p^2 W_1+p^4 \frac{\partial}{\partial s}+p^5 \frac{\partial}{\partial P}+p^6 \frac{\partial}{\partial Q}\right)\\&=p^2(a_5+Q a_1)-a_2(p^5+Q p^1)+a_6-a_3 p^6=0
\end{align*}
We thus have that
\begin{equation*}
    H(\langle X_1\rangle)=\{(a_1,\dots,a_6)\in \mathbb{R}^6\mid P a_2+Q a_3+a_4=Q p^2 a_1-(p^5+Q p^1)a_2-p^6 a_3+p^2 a_5+a_6=0\},
\end{equation*}
which gives that $\dim H((F_1)_{\bar{m}})=4$ for all $(F_1)_{\bar{m}}\in G_1(\overline{A},\Omega)$. Hence, $(E_1)_{z}$ is K{\"a}hler-regular.

We now check that $(E_2)_{z}$ is K{\"a}hler-regular. If $(F_2)_{\bar{m}}=\langle Y_1,Y_2\rangle$ is any $2$-dimensional integral element, then note that
\begin{equation*}
    \beta\wedge \sigma_{11}(v,Y_1,Y_2)=\beta(v,Y_1)\sigma_{11}(Y_2)-\beta(v,Y_2)\sigma_{11}(Y_1)+\beta(Y_1,Y_2)\sigma_{11}(v)=\beta(Y_1,Y_2)\sigma_{11}(v),
\end{equation*}
for every $\beta\in \Lambda^2(\overline{A})$, hence $\beta\wedge \sigma_{11}(v,Y_1,Y_2)=0$ for all $\beta$ if and only if $\sigma_{11}(v)=0$. Also, note that
\begin{align*}
    \alpha\wedge \delta \sigma_{11}(v,Y_1,Y_2)&=\alpha(v)\delta \sigma_{11}(Y_1,Y_2)-\alpha(Y_1)\delta \sigma_{11}(v,Y_2)+\alpha(Y_2)\delta \sigma_{11}(v,Y_1)\\&=-\alpha(Y_1)\delta \sigma_{11}(v,Y_2)+\alpha(Y_2)\delta \sigma_{11}(v,Y_1),
\end{align*}
which is zero for all $\alpha\in \Lambda^1(\overline{A})$ if and only if $\delta \sigma_{11}(v,Y_1)=\delta \sigma_{11}(v,Y_2)=0$. This gives that
\begin{multline*}
    H(\langle Y_1,Y_2\rangle)=\Bigg\{(a_1,\dots,a_6)\in \mathbb{R}^6\mid P a_2+Q a_3+a_4\\=Q p_1^2 a_1-(p_1^5+Q)a_2-p_1^6 a_3+p_1^2 a_5+a_6=Q p_2^2 a_1-p_2^5 a_2-p_2^6 a_3+p_2^2 a_5+a_6=0\Bigg\}
\end{multline*}
Note that the last two equations are linearly dependent if and only if $p_1^2=p_2^2$, $p_1^6=p_2^6$ and $p_1^5+Q=p_2^5$. We therefore have that
\begin{equation*}
    \dim H((F_2)_{\bar{m}})=\begin{cases} 4,\quad \mathrm{if}~p_1^2=p_2^2,p_1^6=p_2^6,p_1^5+Q=p_2^5, \\ 3,\quad \mathrm{otherwise}.
        
    \end{cases}
\end{equation*}

Note that $\dim(H((E_2)_z))=3$, since $p_1^2((E_2)_z)=0$ and $p_2^2((E_2)_z)=1$. It follows that $r$ is a constant function on a neighborhood of $(E_2)_z$, hence $(E_2)_z$ is K{\"a}hler-regular.

Therefore, by the Cartan--K{\"a}hler theorem, there exists an integral manifold $j\colon X\to \overline{M}$ such that $J((\mathcal{L}^j\overline{A})_x)=E_z$ where $x\in X$ is a point such that $z=j(x)$ for all points of the form $z=(0,0,c_1,c_2,0)$, where $c_1\neq 0$ and $c_2\in \mathbb{R}$. 

As before, the Cartan--K{\"a}hler theorem only guarantees the existence of an integral manifold, but does not give one explicitly. We now claim that
\begin{equation}\label{eqn:specialj}
    j\colon M\longrightarrow \overline{M},\quad (t,w)\longmapsto \left(t,w,s=c_1 e^{-h(w)},P=c_1 e^{-h(w)} h'(w),Q=0\right),
\end{equation}
where $h(w)\coloneqq \frac{c_2}{c_1}w+\frac{d_2 c_1+c_2^2}{2c_1^2} w^2$ is such an integral manifold. For a generic element $\bar{a}_{\bar{m}}=a_1\Gamma_0+a_2 W_1+a_3 H_1+a_4 \frac{\partial}{\partial s}+a_5 \frac{\partial}{\partial P}+a_6 \frac{\partial}{\partial Q}\in \overline{A}$, one has that
\begin{equation*}
    \bar{\rho}(\bar{a}_{\bar{m}})=a_1\frac{\partial}{\partial t}+a_2 \frac{\partial}{\partial w}+a_4 \frac{\partial}{\partial s}+a_5 \frac{\partial}{\partial P}+a_6 \frac{\partial}{\partial Q}.
\end{equation*}
A generic element $v_m=b_1 \frac{\partial}{\partial t}+b_2\frac{\partial}{\partial w}\in \mathrm{T}M$ transforms under (the tangent map of) $j$ via
\begin{equation*}
    \mathrm{T}j(v_m)=b_1\frac{\partial}{\partial t}+b_2 \frac{\partial}{\partial w}-b_2 c_1 e^{-h(w)}h'(w) \frac{\partial}{\partial s}+b_2 c_1 e^{-h(w)}\left(h''(w)-(h'(w))^2\right)\frac{\partial}{\partial P}.
\end{equation*}
Points in $\mathcal{L}^j \overline{A}$ must satisfy the equations $a_1=b_1$, $a_2=b_2$, $a_4=-b_2 c_2 e^{-h(w)}h'(w)$, $a_5=b_2 c_1 e^{-h(w)}\left(h''(w)-(h'(w))^2\right)$ and $a_6=0$. Hence, a basis of $\mathcal{L}^j \overline{A}$ is given by
\begin{equation*}
    \left\{\left(\Gamma_0,\frac{\partial}{\partial t}\right),\left(W_1-c_1 e^{-h(w)}h'(w)\frac{\partial}{\partial s}+c_1 e^{-h(w)}\left(h''(w)-(h'(w))^2\right)\frac{\partial}{\partial P},\frac{\partial}{\partial w}\right),(H_1,0)\right\}.
\end{equation*}
Applying $J$ and using the values of $h$, $h'$ and $h''$ at $0$ gives
\begin{equation*}
    J((\mathcal{L}^j \overline{A})_{(0,0)})=\spn\left\{\Gamma_0(j(0,0)),W_1(j(0,0))-c_2 \frac{\partial}{\partial s}\Big|_{j(0,0)}+d_2 \frac{\partial}{\partial P}\Big|_{j(0,0)},H_1(j(0,0))\right\}=E_z,
\end{equation*}
as required.

\begin{remark}
The natural choice of $1$-integral flag $(E_1)_z\subset (\widetilde{E}_2)_z\subset E_z$, where
\begin{equation*}
    (E_1)_z\coloneqq \spn\{H_1(z)\},\quad (\widetilde{E}_2)_z\coloneqq \spn\{\Gamma_0(z),H_1(z)\},
\end{equation*}
does not work here, because $(\widetilde{E}_2)_z$ is not K{\"a}hler-regular. Indeed, $\dim H((\widetilde{E}_2)_z)=4$. If $V=\{(t,w,s,P,Q,p_i^j)\mid p_1^2=p_2^2,p_1^6=p_2^6,p_1^5+Q=p_2^5\}$, then note that $\dim(V)\neq \dim(G_2(\overline{A},\Omega))$. Hence any subset of $V$ is not open, and therefore $(E_2)_z$ is not K{\"a}hler-regular.
\end{remark}

The integral manifold $j$ in \eqref{eqn:specialj} has a further property that is not automatically guaranteed by the Cartan--K{\"a}hler Theorem. It is, in fact, a section of the fiber bundle $p\colon \overline{M}\to M$ and satisfies, for this reason, an \emph{independence condition}, namely $j^{\ast}(\mathrm{d}t\wedge \mathrm{d}w^1\wedge \dots\wedge \mathrm{d}w^n)\neq 0$, where $(t,w^i)$ are the coordinates of $M=\mathbb{R}\times \mathfrak{g}$ (see the next section for more details). This means that $j$ is indeed of the ``special type'' which induces a multiplier matrix for the invariant inverse problem (see the reference we had made earlier to Thm. 2 of \cite{mestdag2026edsla}).

%% file: sections/outlook.tex
We end the paper with some different directions for future research.

First, it may be possible to extend \emph{Cartan's test} to (transitive) Lie algebroids. For EDS on manifolds, it expresses the condition of integrability in terms of certain integers (called \emph{characters}).  With such an extension in place, one could then attempt to extend the Frobenius theorem to transitive Lie algebroids in the analytic category (in other words, showing that all integral elements $E_z\subset \overline{A}_z$ of dimension $N-s$ are $\dim(\ker(\bar{\rho}_z))$-ordinary for an EDS on a transitive Lie algebroid of rank $N$, generated algebraically by $s$ linearly independent $1$-forms). 

Second, in the classical theory of EDS on manifolds, one often works with an extra assumption called an \emph{independence condition}. This means that, given a choice of a differential $n$-form $\Omega$ on $\overline{M}$, integral manifolds $i\colon M\to \overline{M}$ must additionally be $n$-dimensional, and they must satisfy $i^{\ast}(\Omega)\neq 0$ at each point of $M$ (besides the condition that $i^{\ast}\theta=0$ for any $\theta\in \mathcal{I}\subset \Omega^{\ast}(\overline{M})$). We have explained in \cite{mestdag2026edsla} that if $p\colon \overline M \to M$ is a fiber bundle, any integral manifold $i\colon M\to \overline M$ that is also a section implicitly satisfies a certain independence condition. For the specific case of an EDS with an independence condition, finding the characters in Cartan's test is easy using \emph{tableaux}. Roughly speaking, a tableau is a matrix encoding such an EDS. It also serves as a computational tool, to compute the characters using only linear algebra \cite{ivey2016cartan,mckay2017eds}. We believe that extending tableaux to transitive Lie algebroids will be very useful in determining existence of integral manifolds. A natural starting point would be to study extensions of so-called \emph{linear Pfaffian systems} (a linear EDS generated by only one-forms together with an independence condition) in the Lie algebroid setting, since the theory of tableaux is simplest in this situation (in the classical theory of EDS).

A third possible extension of this work is motivated by our discussion on the invariant inverse problem of the calculus of variations. While in \Cref{sec:inv-inverse} we have shown that our version of the Cartan--K{\"a}hler theorem is \emph{consistent} with the fact that a $G$-invariant Lagrangian exists for the geodesic spray of the canonical connection on the Lie group $G=\mathbb{R}$, we were not able to \emph{conclude} using the Cartan--K{\"a}hler theorem that a reduced multiplier matrix exists. This is because the existence of a reduced multiplier requires the additional condition that there exists an integral manifold $j\colon M\to \overline{M}$ which is a \emph{section} of $p\colon \overline{M}\to M$ (and therefore satisfies an independence condition), which is not guaranteed by the Cartan--K{\"a}hler theorem. We would therefore like to devise techniques to solve this problem using our version of the Cartan--K{\"a}hler theorem. Since this is an EDS with an independence condition, we suspect that an extension of the theory of tableaux will be very helpful for this.

If such an approach is successful, we may be able to study more nontrivial cases of this problem. Following the notation in \cite{mestdag2026edsla}, the natural starting point for study is the case where $\langle \Sigma^0\rangle_{\mathrm{alg}}$ is a differential ideal. In the future, a classification structure (analogous to Douglas' classification \cite{douglas1941classification}, or the classification in \cite{do2016invprob}) for this problem may be identified.

%% file: sections/comparison.tex
The current framework defines integral manifolds by starting from a given transitive Lie algebroid $\bar{\tau}\colon \overline{A}\to \overline{M}$. We then constructed a new Lie algebroid $\tau\colon A\to M$ by setting $A\coloneqq \mathcal{L}^i \overline{A}$ to be the $i$-prolongation Lie algebroid, where $i\colon M\to \overline{M}$ is an immersion. We then said that $i\colon M\to \overline{M}$ is an integral manifold if $I^{\ast}\theta=0$ for all $\theta\in \mathcal{I}$, where $I$ is the $i$-induced map. Integral manifolds in this framework are not required to satisfy any independence condition, i.e., if $n$ is the dimension of the integral manifolds one must find, there is no requirement that $i^{\ast}(\Omega)\neq 0$ for some $\Omega\in \Omega^n(\overline{M})$.

Howewer, in a previous paper \cite{mestdag2026edsla}, we defined integral manifolds by instead beginning with the Lie algebroid $\tau\colon A\to M$. We then constructed $\bar{\tau}\colon \overline{A}\to \overline{M}$ by defining $\overline{A}\coloneqq \mathcal{L}^p A$, where $p\colon \overline{M}\to M$ is a fiber bundle. In this framework, $i$ is no longer an arbitrary immersion, but a section of $p$. In order to distinguish the two definitions, we use $s$ instead of $i$. The definition given there is as follows:
\begin{definition}
Let $s\colon M\to \overline{M}$ be a section of $p$. We say that $s$ is an \emph{integral manifold} of a differential ideal $\mathcal{I}$ if for all $\theta\in \mathcal{I}\subset \Lambda^{\ast}(\mathcal{L}^p A)$, the equation $S^{\ast}\theta=0$ holds in $\Lambda^{\ast}(A)$, where $S$ is the $s$-induced section, i.e.,
\begin{equation*}
    S(a_m)\coloneqq ( a_m, \mathrm{T}s(\rho(a_m))).
\end{equation*}
\end{definition}
This approach generalizes the definition of an integral manifold of an EDS with an independence condition. Indeed, if $p\colon \overline{M}\to M$, $(x^i,y^{\mu})\mapsto (x^i)$ is a fiber bundle, then any section $s$ of $p$ satisfies $s^{\ast}(\mathrm{d}x^1\wedge \dots \wedge \mathrm{d}x^n)\neq 0$.

We now claim that in fact, $I=S$ in the case where $i$ is a section of a fiber bundle. Recall that $I\colon A=\mathcal{L}^i \overline{A}\to \overline{A}$ was given by $(\bar{a}_{\bar{m}},w_m)\mapsto \bar{a}_{\bar{m}}$. To this end, we first prove the following lemma:
\begin{lemma}
Let $\tau\colon A\to M$ be a Lie algebroid. Then $A\simeq \mathcal{L}^{s}(\mathcal{L}^p A)$.
\end{lemma}
\begin{proof}
We write the identification explicitly. Note that $\mathcal{L}^p A$ consists of pairs $(a_m,\bar{v}_{\bar{m}})\in A\times \mathrm{T}\overline{M}$ such that $\rho(a_m)=\mathrm{T}p(\bar{v}_{\bar{m}})$. The anchor of $\mathcal{L}^p A$ is
\begin{equation*}
    \bar{\rho}\colon \mathcal{L}^p A\longrightarrow \mathrm{T}\overline{M},\quad (a_m,\bar{v}_{\bar{m}})\longmapsto \bar{v}_{\bar{m}}.
\end{equation*}
We now determine $\mathcal{L}^s (\mathcal{L}^p A)$. An element of this space is a pair $((a_m,\bar{v}_{\bar{m}}),w_m)\in \mathcal{L}^p A\times \mathrm{T}M$ subject to $\bar{v}_{\bar{m}}=\mathrm{T}s(w_m)$. Now, note that
\begin{equation*}
    \rho(a_m)=\mathrm{T}p(\bar{v}_{\bar{m}})=\mathrm{T}p(\mathrm{T}s(w_m))=w_m.
\end{equation*}
Hence, the following is an isomorphism
\begin{equation*}
    \Phi\colon \mathcal{L}^s(\mathcal{L}^p A)\longrightarrow A,\quad ((a_m,\mathrm{T}s(\rho(a_m))),\rho(a_m))\longmapsto a_m.\qedhere
\end{equation*}
\end{proof}
We now prove the equivalence of the two frameworks in the case when one imposes that $i$ is a section.
\begin{proposition}
Suppose that $\bar{\tau}\colon \overline{A}\to \overline{M}$ is a transitive Lie algebroid. If $i\colon M\to \overline{M}$ is a section of a fiber bundle $p\colon \overline{M}\to M$, then $I=S$.
\end{proposition}
\begin{proof}
Note that
\begin{equation*}
    S(a_m)=S(\bar{a}_{\bar{m}},w_m)=((\bar{a}_{\bar{m}},w_m),\mathrm{T}s(\rho(\bar{a}_{\bar{m}},w_m)))=((\bar{a}_{\bar{m}},w_m),\mathrm{T}s(w_m))=(a_m,\bar{v}_{\bar{m}})=\bar{a}_{\bar{m}},
\end{equation*}
where we have used that $\bar{v}_{\bar{m}}=\mathrm{T}s(w_m)$, and the notation $a_m=(\bar{a}_{\bar{m}},w_m)$ and $\bar{a}_{\bar{m}}=(a_m,\bar{v}_{\bar{m}})$.
\end{proof}
As a result, also $I^{\ast}=S^{\ast}$.